\documentclass{article}
\usepackage[english]{babel}
\usepackage{amsmath,amsthm,amssymb,xcolor,tabularx}
\usepackage[utf8]{inputenc}
\usepackage{float}
\usepackage{xcolor}
\definecolor{green1}{cmyk}{1,0,1,0}
\usepackage{caption}
\usepackage{booktabs}
\usepackage{array}
\usepackage{makecell}
\usepackage{abstract}
\usepackage{multirow}
\usepackage{subcaption}
\usepackage[anchorcolor=blue, linkcolor=red, urlcolor=black, citecolor=blue, pagebackref=true, colorlinks=true]{hyperref}
\usepackage[toc,page,header]{appendix}
\usepackage{minitoc}
\usepackage{bm}
\usepackage{enumitem}
\usepackage[numbers,sort&compress]{natbib}
\usepackage{graphicx}
\usepackage{enumitem}
\setlist[itemize]{topsep=2pt, leftmargin=1em}
\setlist[enumerate]{leftmargin=1.8em, label=(\roman*)}

\usepackage{aliascnt}
\usepackage[ruled]{algorithm}
\usepackage{algorithmic}
\makeatother

\usepackage{geometry}
\theoremstyle{definition}
\newtheorem{theorem}{Theorem}[section]
\newaliascnt{lemma}{theorem}
\newtheorem{lemma}[lemma]{Lemma}
\aliascntresetthe{lemma}
\newaliascnt{corollary}{theorem}

\aliascntresetthe{corollary}
\newaliascnt{proposition}{theorem}
\newtheorem{proposition}[proposition]{Proposition}
\aliascntresetthe{proposition}

\newaliascnt{remark}{theorem}
\newtheorem{remark}[remark]{Remark}
\aliascntresetthe{remark}
\newtheorem{assumption}{Assumption}[section]
\newaliascnt{example}{theorem}
\newtheorem{example}[example]{Example}
\aliascntresetthe{example}

\usepackage[capitalize]{cleveref}
\crefname{algorithm}{Algorithm}{Algorithms}
\Crefname{algorithm}{Algorithm}{Algorithms}
\crefname{algocf}{Algorithm}{Algorithms}
\Crefname{algocf}{Algorithm}{Algorithms}
\crefname{algo}{Algorithm}{Algorithms}
\Crefname{algo}{Algorithm}{Algorithms}
\crefname{theorem}{Theorem}{Theorems}
\Crefname{theorem}{Theorem}{Theorems}
\crefname{lemma}{Lemma}{Lemmas}
\Crefname{lemma}{Lemma}{Lemmas}
\crefname{corollary}{Corollary}{Corollaries}
\Crefname{corollary}{Corollary}{Corollaries}
\crefname{proposition}{Proposition}{Propositions}
\Crefname{proposition}{Proposition}{Propositions}
\crefname{assumption}{Assumption}{Assumptions}
\Crefname{assumption}{Assumption}{Assumptions}
\crefname{example}{Example}{Examples}
\Crefname{example}{Example}{Examples}
\crefname{remark}{Remark}{Remarks}
\Crefname{remark}{Remark}{Remarks}
\crefname{figure}{Figure}{Figures}
\Crefname{figure}{Figure}{Figures}

\DeclareRobustCommand{\namerefc}[2]{\hyperref[#2]{\textcolor{#1}{\nameref*{#2}}}}

\newcommand{\sosgm}{\texttt{SOSGM}}
\newcommand{\ov}{\namerefc{purple}{alg:osgm-svrg}}
\newcommand{\osgd}{\namerefc{purple}{alg:sosgm}}

\global\long\def\vertiii#1{\left\vert \kern-0.25ex  \left\vert \kern-0.25ex  \left\vert #1\right\vert \kern-0.25ex  \right\vert \kern-0.25ex  \right\vert }

\global\long\def\argmin{\operatornamewithlimits{arg\,min}}

\global\long\def\and{\mathrm{and}}

\global\long\def\Rbb{\mathbb{R}}

\newcommand{\assign}{:=}
\newcommand{\tmop}[1]{\ensuremath{\operatorname{#1}}}

\newcolumntype{Y}{>{\centering\arraybackslash}X}

\newcommand{\os}{\texttt{OSGM}}

\newcommand{\ovh}{\texttt{OSGM-SVRG}}
\newcommand{\osgdh}{\texttt{OSGM-SGD}}
\newcommand{\osk}{\texttt{OSGM-SketchySVRG}}

\newcommand{\svrg}{\texttt{SVRG}}
\newcommand{\saga}{\texttt{SAGA}}
\newcommand{\kat}{\texttt{L-Katyusha}}
\newcommand{\ssvrg}{\texttt{SketchySVRG}}
\newcommand{\ssaga}{\texttt{SketchySAGA}}
\newcommand{\skat}{\texttt{SketchyKatyusha}}
\newcommand{\promise}{\texttt{PROMISE}}
\newcommand{\sh}{\texttt{SVRG-HBM}}
\newcommand{\adam}{\texttt{Adam}}
\newcommand{\sgd}{\texttt{SGD}}
\newcommand{\hbm}{\texttt{HBM}}
\newcommand{\agd}{\texttt{SGDN}}
\newcommand{\hd}{\texttt{HD}}
\newcommand{\sgdhd}{\texttt{SGD-HD}}
\newcommand{\sgdnhd}{\texttt{SGDN-HD}}
\newcommand{\adamhd}{\texttt{Adam-HD}}

\date{}
\begin{document}

\author{Wanyu Zhang\footnotemark[1]\quad 
Wenzhi Gao\footnotemark[1] \quad
Yinyu Ye\footnotemark[1] \quad
Madeleine Udell\footnotemark[1]
}

\renewcommand{\thefootnote}{\fnsymbol{footnote}}
\footnotetext[1]{Stanford University. \texttt{\{zwanyu,gwz,yyye,udell\}@stanford.edu}}
\renewcommand{\thefootnote}{\arabic{footnote}}

\title{Stochastic Gradient Methods with Online Scaling}

\maketitle

\begin{abstract}
    
This paper introduces Stochastic Online Scaled Gradient Methods ({\sosgm}), a generalization of the recently developed adaptive preconditioning framework in \cite{gao2025gradient,chu2025gradient} to stochastic optimization.
Under standard assumptions, we establish convergence guarantees for {\sosgm} using large batchsize or variance reduction. {\sosgm} is compatible with popular diagonal and/or low-rank preconditioners as well as heavy-ball momentum, while maintaining memory and computation cost comparable to \texttt{Adam}.
Extensive numerical experiments demonstrate the strong empirical performance of {\sosgm}. Using a diagonal preconditioner, {\sosgm} and its variants substantially outperform existing adaptive first-order methods across a range of statistical learning tasks.

\end{abstract}

\section{Introduction}
\label{sec:intro}

Optimization on large-scale datasets uses stochastic gradient descent ({\sgd}) and its variants. However, {\sgd} typically achieves only sublinear convergence rates on smooth, strongly convex objectives due to non-vanishing noise in the stochastic gradient. 
Although variance-reduction (VR) methods, such as {\svrg} \cite{johnson2013accelerating}, {\saga} \cite{defazio2014saga}, and \texttt{Katyusha} \cite{allen2018katyusha}, address this issue and offer improved convergence rates for convex problems, they still perform poorly on ill-conditioned data, which are pervasive in real-world problems \cite[Table 2]{frangella2024sketchysgd}. 
Ill-conditioning forces these methods to use conservative stepsizes for stability, limiting progress even with careful tuning.

Methods that use second-order information, such as Newton's method or \texttt{L-BFGS}, can converge quickly even on ill-conditioned data. 
While these methods do not scale to large datasets, 
many stochastic second-order methods have been proposed to deliver better performance than first-order methods. 
Most use subsampling-based approximations to the Hessian, 
which either directly compute the search direction using the inverse of a subsampled Hessian \cite{roosta2019sub, bollapragada2019exact, derezinski2022stochastic}, 
or use subsampled Hessian to stabilize \texttt{L-BFGS}-style updates \cite{byrd2016stochastic,moritz2016linearly, gower2016stochastic}.
However, the former can be computationally heavy due to repeated linear-system solves involving subsampled Hessians, and both suffer from unstable curvature estimates.

Sketching-based preconditioners offer a complementary approach to handling ill conditioned problems.
Notably, \texttt{SketchySGD} \cite{frangella2024sketchysgd} and {\promise} \cite{frangella2024promise} propose using scalable sketching methods to construct randomized low-rank preconditioners. 
These methods can boost the performance of {\svrg}, {\saga}, and \texttt{Katyusha}, among others, and work well for large-scale dense data.
However, for sparse data, the cost of storing and applying a low rank preconditioner compares poorly to the cost of storing or applying the data matrix, and so low-rank preconditioners are not recommended.

On a different front, in deep learning,
adaptive gradient methods are widely used; many can be viewed as applying adaptive diagonal preconditioners,
such as \texttt{AdaGrad} \cite{duchi2011adaptive}, \texttt{RMSProp} \cite{hinton2012neural}, and {\adam} \cite{kingma2014adam}.
These methods deliver faster convergence than {\sgd} in practice.
These methods are often strong empirically, but their update rules are typically heuristic and do not directly optimize a principled objective for choosing the preconditioner.
While these methods accelerate training in practice, the diagonal preconditioners of adaptive optimizers like {\adam} are not designed to explicitly reduce the condition number of the problem.

Providing a theoretical foundation for stepsize adaptation, online scaled gradient methods ({\os}) \cite{gao2025gradient,chu2025gradient} offer a deterministic framework to adjust the stepsize / preconditioner,
where the problem of choosing a stepsize is formulated as an online decision-making problem tackled by online learning algorithms.
Theoretically, {\os} achieves convergence results that are asymptotically no worse than the optimal stepsize.
By learning a matrix stepsize online, {\os} can adapt to the local geometry of the loss landscape and mitigate the effect of ill-conditioning.

The key missing piece is a stochastic counterpart with finite-sum convergence guarantees:
\begin{center}
  \emph{Can we design efficient adaptive stochastic gradient methods that are robust to ill-conditioning?}
\end{center}
This work develops an affirmative answer to the question by introducing Stochastic Online Scaled Gradient Methods ({\sosgm}), which generalize {\os} to the stochastic optimization setting. 
{\sosgm} treats matrix stepsize selection as an online decision problem and include instantiations {\osgd} and {\ov}. 
{\sosgm} is compatible with scalar, diagonal, or matrix stepsizes, and optional heavy-ball momentum; for diagonal stepsizes, the memory and per-iteration cost are comparable to {\adam}.

\paragraph{Structure of the paper} This paper is organized as follows. 
\Cref{sec:sosgm} introduces {\sosgm}, an {\os}-style framework to tune the matrix stepsize. 
\Cref{sec:large-batchsize} develops {\osgd} to learn the stepsize of {\sgd}, and establishes high-probability convergence guarantees under a gradient-norm condition. 
\Cref{sec:vr} presents {\ov}, which adjusts the stepsize of {\svrg} in the outer loop. We prove the linear convergence for strongly convex objectives. 
\Cref{sec:exp} showcases empirical performance of {\sosgm} on statistical learning and deep learning problems.

\paragraph{Naming conventions} 
Throughout the paper, colored, name-referenced algorithm names (such as \osgd) denote the theoretically analyzed {\sosgm} algorithms; uncolored, plain-text names (such as {\osk}) denote the practical variants used in experiments.
We use \emph{stepsize} to refer to the stepsize/preconditioner in the optimization update, and we reserve \emph{learning rate} for the stepsize used by the online gradient descent (hypergradient) update of the stepsize/preconditioner.

\subsection{Related work}

\paragraph{Hypergradient descent} The hypergradient descent ({\hd}) method was first presented in \cite{almeida1999parameter} as a heuristic to accelerate {\sgd}. 
Similar ideas have been explored independently, 
such as incremental delta-bar-delta \cite{sutton1992adapting}, 
stochastic meta-descent \cite{schraudolph1999local}, 
and other adaptive schemes \cite{jacobs1988increased,mahmood2012tuning}. 
This method was later rediscovered and called hypergradient descent by \cite{baydin2018hypergradient}, 
which also extended the {\hd} idea to tune {\sgd} and {\adam}, 
with experiments on optimizing statistical learning problems and neural networks. 
Theoretical understanding of {\hd} developed later, starting with
\cite{rubio2017convergence} which analyzed convergence of {\hd} for deterministic gradient descent on convex quadratic objectives, and continuing with \cite{powerball2023} which analyzed convergence of a particular stochastic optimizer under the {\hd} stepsize.

A more complete explanation for the empirical advantage of {\hd}
has recently been developed in a series of work \cite{gao25a,chu2025provable,gao2025gradient,chu2025gradient}, 
which establishes the {\os} framework to choose a matrix stepsize by online learning.
{\os} uses online gradient descent on different loss functions, and has strong trajectory-based convergence guarantees in the deterministic setting. 
This paper adapts {\os} for stochastic gradient methods.

\paragraph{Stochastic second-order methods} 
A variety of stochastic analogs of classical Newton-type methods have been developed to address large-scale ill-conditioned optimization problems. 
Several works propose subsampling-based Hessian approximations, 
which either directly compute the search direction using the inverse of a subsampled Hessian \cite{roosta2019sub, bollapragada2019exact, derezinski2022stochastic}, 
or use a subsampled Hessian to stabilize \texttt{L-BFGS}-style updates \cite{byrd2016stochastic,moritz2016linearly, gower2016stochastic}.
Several of these methods require expensive repeated linear-system solves involving subsampled Hessians, and most require a large batchsize or sufficiently good initialization for convergence guarantees.

Compared to stochastic second-order methods, {\sosgm} adapts to ill-conditioning without forming Hessian approximations or solving linear systems. Instead, {\sosgm} learns a matrix stepsize by online gradient descent, while retaining per-iteration costs comparable to {\adam} when instantiated with diagonal stepsizes.

\paragraph{Preconditioned stochastic gradient methods}
This line of work designs explicit preconditioners for stochastic gradient methods. 
\cite{liu2019acceleration} accelerated {\svrg} and \texttt{Katyusha} by applying inexact preconditioners derived from approximate Hessian solves. 
\cite{frangella2024sketchysgd, frangella2024promise} 
used randomized low-rank preconditioners to improve the convergence of {\sgd} and  methods such as {\svrg}, {\saga}, and \texttt{Loopless Katyusha} ({\kat} \cite{kovalev2020don}). 
They established global linear convergence with a constant batchsize,
and demonstrated excellent performance through experiments on large-scale ill-conditioned machine learning problems. 
\cite{sun2025sapphire} showed how to use randomized low-rank preconditioners in the context of a stochastic proximal gradient methods, with impressive empirical results for regularized statistical learning problems like LASSO and elastic net.

\paragraph{Adaptive stochastic gradient methods} 
Adaptive gradient methods are widely used in stochastic optimization, particularly for training neural networks. 
Many of these methods can be viewed as employing a diagonal preconditioner learned from previous gradients, 
such as \texttt{AdaGrad} \cite{duchi2011adaptive}, {\adam} \cite{kingma2014adam}, \texttt{AdaHessian} \cite{yao2021adahessian}, \texttt{Lion} \cite{chen2023symbolic}, \texttt{Sophia} \cite{liu2024sophia}, and \texttt{Adafactor} \cite{shazeer2018adafactor}. 
More recently, optimizers like \texttt{Shampoo} \cite{gupta2018shampoo}, \texttt{K-FAC} \cite{martens2015optimizing}, and \texttt{SOAP} \cite{vyas2025soap} apply structured (non-diagonal) preconditioners and show promising performance in training large language models.

Theoretical work on adaptive stochastic gradient methods develops principled stepsize selection procedures with provable convergence. 
For example, \cite{Malitsky20} proposed an adaptive stepsize based on local curvature estimates, with convergence rates depending on local geometry; 
its counterpart for stochastic gradient methods was studied in \cite{aujol2025stochastic}. 
Line-search rules for {\sgd} were explored in \cite{vaswani2019painless,vaswani2025armijo}, and \cite{vaswani2025armijo} further established faster convergence than standard {\sgd} under interpolation-type conditions. Finally, \cite{loizou2021stochastic} analyzed the Polyak stepsize for {\sgd} and provided convergence guarantees in both convex and non-convex settings. 
Recently, adaptive methods without knowing the problem parameters have been proposed, including D-adaptation for unknown distance-to-optimality \cite{defazio2023learning} and schedule-free optimizers that avoid dependence on a pre-specified training horizon \cite{defazio2024road}.  
However, most work in this line focuses on \emph{scalar} stepsizes and therefore does not improve the condition number of the problem.

\subsection{Notations}

We use $\| \cdot \|$ to denote the Euclidean norm of vectors or the operator norm of matrices, and $\langle \cdot, \cdot
\rangle$ to denote the Euclidean or Frobenius inner product. The notation
$\| A \|_F \assign \sqrt{\sum_{i j} a_{i j}^2}$ denotes the matrix Frobenius norm. Given a closed convex set $\mathcal{P}$, $\Pi_{\mathcal{P}} [\cdot]$ denotes the orthogonal projection onto $\mathcal{P}$; $\tmop{dist} (P, \mathcal{P}) \assign \| P - \Pi_{\mathcal{P}}
[P] \|_F$ denotes the distance between a point $P$ and set 
$\mathcal{P}$; Given a vector $v\in \mathbb{R}^d$, $\text{Diag}(v)$ denotes the diagonal matrix with elements of $v$ on its diagonal.
We use $\mathcal{X}^{\star} = \{ x : f (x) = f^{\star} \}$ to denote the optimal
set of $f$; $\tmop{diam} (\mathcal{P}) = \max_{X, Y \in \mathcal{P}}  \| X
- Y \|_F$ denotes the diameter of the set $\mathcal{P}$ in Frobenius norm. 
A function $f$ is $L$-smooth if it satisfies $\| \nabla f(x) - \nabla f(y) \| \leq L \|x - y\|$ for all $x, y \in \mathbb{R}^d$. 
We use superscripts $x^k$ to index algorithm iterates and subscripts $P_k$ to index the stepsize sequence. The notation $\mathbb{S}^{n}$ denotes the set of $n$ by $n$ symmetric matrices. 
For asymptotic complexity, $\tilde{\mathcal{O}}(\cdot)$ hides polylogarithmic factors.

\section{Stochastic online scaled gradient methods}
\label{sec:sosgm}

To motivate the stochastic setting, this section first reviews the deterministic {\os} framework and the challenges that arise when gradients are noisy, then introduces {\sosgm}, a general framework for learning the matrix stepsize of stochastic gradient methods.

\subsection{Deterministic online scaled gradient methods}
Consider the unconstrained minimization of a deterministic smooth convex function
\( \min_{x\in \Rbb^d} ~ f(x)\)
with preconditioned gradient descent:
\[
x^{k+1} = x^k - P_k \nabla f(x^k),
\]
where $P_k\in \mathbb{R}^{d \times d}$ is a matrix stepsize that can be scalar ($P_k = \alpha_k I$ for $\alpha_k \in \mathbb{R}$), diagonal ($P_k = \mbox{diag}(v_k)$ for $v_k \in \mathbb{R}^d$) or a general matrix. 
{\os} \cite{gao2025gradient,chu2025gradient} is a framework that uses online learning to adjust stepsize $\{P_k\}$.

To motivate {\os}, consider the standard analysis of a linearly convergent method. The usual goal is to prove a uniform one-step contraction,
\(
\tfrac{f(x^{k+1})-f^\star}{f(x^k)-f^\star}\le 1-\tfrac{1}{\kappa}
\),
for all $k$. Multiplying these per-iteration bounds gives
\[
\tfrac{f(x^{K}) - f^\star}{f(x^0) - f^\star}=\textstyle \prod_{k=0}^{K-1} \tfrac{f(x^{k+1})-f^\star}{f(x^k) - f^\star} \leq (1-\tfrac{1}{\kappa})^K.
\]

Different from the above analysis, {\os} first chains the progress and gives an upper bound by the arithmetic-geometric mean inequality
\[
\tfrac{f(x^{K}) - f^\star}{f(x^0) - f^\star}=\textstyle \prod_{k=0}^{K-1} \tfrac{f(x^{k+1})-f^\star}{f(x^k) - f^\star} \leq (\tfrac{1}{K} \textstyle \sum_{k=0}^{K-1} \tfrac{f(x^{k+1})-f^\star}{f(x^k) - f^\star} )^K .
\]
From this, faster convergence is achieved by minimizing the average contraction ratio. Define $r_{x^k}(P_k) \assign \tfrac{f(x^{k} - P_k \nabla f(x^k))-f^\star}{f(x^k) - f^\star}$; it suffices to choose the stepsizes $\{P_k\}$ sequentially to minimize $\tfrac{1}{K} \textstyle \sum_{k=1}^{K} r_{x^k}(P_k)$. In this view, stepsize selection is formulated as an online decision-making problem, thus motivating the use of online learning algorithms to adjust $\{P_k\}$. For example, online gradient descent 
\[
P_{k+1} = P_k - \eta \nabla r_{x^k} (P_k)
\]
yields sublinear regret $\tfrac{1}{K} \textstyle \sum_{k=1}^{K} r_{x^k}(P_k)\leq \tfrac{1}{K} \textstyle \sum_{k=1}^{K} r_{x^k}(\hat P) + \mathcal{O}(\tfrac{1}{\sqrt{K}})$ for any fixed stepsize $\hat P$ \cite{orabona2019modern}. Choose $\hat P$ as a good stepsize $P_{\star}$ that achieves condition number $\kappa_\star < \kappa$ and progress $r_{x^k} (P_{\star}) \leq 1 - \tfrac{1}{\kappa_\star}$ for all $k$. The convergence guarantee follows:
\[
\tfrac{f(x^{K}) - f^\star}{f(x^0) - f^\star}\leq 
(\tfrac{1}{K} \textstyle \sum_{k=0}^{K-1} r_{x^k}(P_k) )^K \leq
(1 - \tfrac{1}{\kappa_\star} + \mathcal{O}(\tfrac{1}{\sqrt{K}}) )^K.
\]
This result suggests that the performance of {\os} is competitive with a good stepsize $P_\star$ when the total number of iterations $K$ is large, even without knowledge of $P_\star$.

The above algorithmic intuition is generalized into the {\os} framework. In each iteration, 
1) stepsize scheduler makes decision $P_k$ from a candidate set $\mathcal{P}$ and proposes an update $x^{k+1/2} = x^k - P_k \nabla f(x^k)$; 
2) the landscape chooses the next iterate $x^{k+1} = \mathcal{M} (x^k, x^{k+1/2})$, for example, with a null step
\begin{equation}\tag{Null step}
  x^{k+1} = \argmin_{x \in \{x^k, x^{k+1/2}\}} \{f(x^k - P_k \nabla f(x^k)), f(x^k)\},
\end{equation}
and provides feedback $\ell_{x^k}(P_k)$ to the scheduler. Two useful feedback functions are
\begin{equation}\label{eq:dtm-feedback}
  \text{ratio} \;\;
  r_{x^k}(P_k) = \tfrac{f(x^k - P_k \nabla f(x^k))-f^\star}{f(x^k) - f^\star} \; 
  \text{or hypergradient } \;\;
  h_{x^k}(P_k) = \tfrac{f(x^k - P_k \nabla f(x^k))-f(x^k)}{\|\nabla f(x^k)\|^2};
\end{equation}
3) scheduler updates the stepsize $P_k$ by an online learning algorithm (such as online gradient descent) with respect to feedback $\ell_{x^k}(P_k)$.

\subsection{Stochastic {\os}}

This paper considers the following finite-sum problem,
\begin{equation}\label{eq:F}
  \min_{x\in \mathbb{R}^d} f (x) \assign \tfrac{1}{n}  \textstyle \sum_{i = 1}^n f_i (x).
\end{equation}
To solve \eqref{eq:F}, consider stochastic preconditioned gradient descent with a heavy-ball momentum ({\hbm}),
\begin{equation*}
  x^{k+1} = x^k - P_k \, g^k + \beta_k (x^k - x^{k-1}),
\end{equation*}
where $g^k$ is an unbiased estimator of the full gradient $\nabla f(x^k)$. 
For example, $g^k$ could be 
\begin{align*}
& \text{({\sgd}) the average of a batch of stochastic gradients,} \qquad   &&\tfrac{1}{|\xi_k|} \textstyle \sum_{i \in \xi_k} \nabla f_i (x^k),\, \text{or} \\
& \text{({\svrg}) a VR estimator with a snapshot $\tilde x$,} \qquad  &&\nabla f_i (x^k) - \nabla f_i (\tilde x) + \nabla f (\tilde x).
\end{align*}
We develop a framework, {\sosgm} (\cref{alg:stoc-osgm}), which adapts {\os} to learn matrix stepsizes for stochastic methods such as {\svrg} and {\sgd}.
At each iteration, {\sosgm} alternates between updating the iterate $x^k$ and updating the stepsize $P_k$.
This framework can be instantiated as the algorithms introduced later:
\begin{itemize}
  \item {\osgd}. Update $P_k$ at every iteration using stochastic ratio or hypergradient feedback.
  \item {\ov}. Update $P_k$ every $m$ inner iterations. Set $\ell_k = 0$ if $\operatorname{mod}(k,m)\neq 0$; otherwise, use deterministic regularized ratio or hypergradient feedback, since the full gradient is available.
\end{itemize}

\begin{algorithm}[h]
\caption{\sosgm}
\label{alg:stoc-osgm}
\begin{algorithmic}[1]
\STATE \textbf{Input:} Initial point $x^0$, initial stepsize $P_0$, candidate stepsize set $\mathcal{P}$, {\os} learning rate $\eta$, \quad momentum sequence $\{\beta_k\}$,
    feedback $\ell_{k}$
\FOR{$k = 0,1,2,\dots$}
        \STATE Compute stochastic gradient estimator $g^k$ on $x^k$

        \STATE Stochastic gradient step
        $x^{k+1}=x^k - P_k \, g^k + \beta_k (x^k - x^{k-1})$

        \STATE Construct feedback $\ell_{k}$ and update stepsize $P_k$ by {\os}
        $P_{k+1} = \Pi_{\mathcal{P}} [P_k - \eta \nabla \ell_{k}(P_k)]$
        
    \ENDFOR
\end{algorithmic}
\end{algorithm}

\subsubsection{Assumptions}

To analyze the performance of these methods, the following assumptions are used throughout the paper.
\begin{assumption}\label{as:smooth}
  For each $i$, $f_i$ is convex and $L$-smooth and $\mu$-strongly convex with $\mu \geq 0$:
  \[
  f_i(y)\geq f_i(x)+\langle \nabla f_i(x), y-x\rangle + \tfrac{\mu}{2}\|y-x\|^2,\quad \forall x,y.
  \]
\end{assumption}
This assumption holds for problems such as regularized generalized linear models.

The stepsize is chosen from a closed convex candidate set $\mathcal{P} \subseteq \mathbb{R}^{d \times d}$. 
It is natural to have the set of preconditioners include $0$, 
as this represents a safe initialization that (we hope will learn a better preconditioner but) does not move the iterate,
and also the theoretically optimal scalar stepsize $\tfrac{1}{L} I \in \mathcal{P}$.
Hence the diameter of $\mathcal{P}$ must be at least $\tfrac{1}{L}$. 
We further impose the assumption that the diameter of $\mathcal{P}$ is in the order of $\tfrac{1}{L}$.
\begin{assumption}\label{as:diam}
  The candidate set of stepsize $\mathcal{P}$ is bounded, $\operatorname{diam}(\mathcal{P}) \leq D$ and $D=\mathcal{O}(\tfrac{1}{L})$.
\end{assumption}

\subsubsection{Failure of naive feedback design}

In stochastic setting where the accurate function values and gradients are expansive or unavailable, one challenge of applying {\sosgm} is the noisy feedbacks. 
We first notice that a naive extension of hypergradient or ratio feedback cannot guarantee convergence.

For {\sgd}, at iteration $k$, we sample a mini-batch $\xi^k$, take a stochastic gradient step on the sampled objective $f_{\xi^k}$. 
To apply {\sosgm}, define the feedback by measuring progress on the same sampled objective, i.e., $f_{\xi^k}(x) \assign \tfrac{1}{|\xi^k|}\sum_{i\in \xi^k} f_i(x)$.
\begin{equation}
  r_{x^k, \xi^k} (P_k) = \tfrac{f_{\xi^k} (x^k - P_k \nabla f_{\xi^k} (x^k)) -
  f^{\star}}{f_{\xi^k} (x^k) - f^{\star}}, \quad
  h_{x^k, \xi^k} (P_k) = \tfrac{f_{\xi^k} (x^k - P_k \nabla f_{\xi^k} (x^k)) -
  f_{\xi^k} (x^k)}{\| \nabla f_{\xi^k} (x^k) \|^2}. \label{eq:hoverfit}
\end{equation}
These feedbacks are \emph{in-sample}: they reuse the same mini-batch $\xi^k$ for both the update and the evaluation of progress.
As a result, they can overfit to the selected mini-batch and fail to reflect progress on the full objective $f$.

We provide a counterexample in \Cref{app:counterexample}, which shows that {\sosgm} with feedback \eqref{eq:hoverfit} does not necessarily converge, even when each $f_{\xi^k}$ is convex and smooth, all sampled objectives share the same minimizer, and $f$ has bounded sublevel sets.

In the counterexample, the learned stepsize and resulting update are (locally) optimal for each sampled objective $f_{\xi^k}$, while still being suboptimal for $f$. 
This suboptimality with respect to $f$ is not captured by the naive feedbacks \eqref{eq:hoverfit}, which evaluate progress only on the selected mini-batch.

We propose two strategies to mitigate the challenge of noisy feedback: out-of-sample feedbacks with large batchsizes for {\osgd} (\Cref{sec:large-batchsize}), and full gradient update in the outerloop for {\ov} (\Cref{sec:vr}).

\section{{\sosgm} with large batchsize}
\label{sec:large-batchsize}

In this section, we consider using {\os} to tune the stepsize of {\sgd},
\[ x^{k + 1} = x^k - P_k \nabla f_{\xi^k} (x^k) . \]
A key difficulty is that naive stochastic extensions of the deterministic feedback functions can fail to converge (see \cref{counterexample}); out-of-sample feedback resolves this issue.
We define out-of-sample feedbacks and present {\osgd}, which has convergence guarantees under the gradient norm condition and large batchsize.

\subsection{\texttt{OSGM-SGD}}

Define the out-of-sample feedbacks by using an independent sample $\zeta^k$, which provides an unbiased estimator for function value decrease:
\begin{equation}
  r_{x^k, \xi^k, \zeta^k} (P_k) = \tfrac{f_{\zeta^k} (x^k - P_k \nabla f_{\xi^k} (x^k)) -
  f^{\star}}{f_{\xi^k} (x^k) - f^{\star}},\quad
  h_{x^k, \xi^k, \zeta^k} (P_k) = \tfrac{f_{\zeta^k} (x^k - P_k \nabla
  f_{\xi^k} (x^k)) - f_{\zeta^k} (x^k)}{\| \nabla f_{\xi^k} (x^k) \|^2} .
  \label{eq:def-h}
\end{equation}
With this feedback design, we define {\osgd} in \cref{alg:sosgm}.

\begin{algorithm}[ht!]
\caption{{\texttt{OSGM-SGD}}}
\label{alg:sosgm}
\begin{algorithmic}[1]
    \STATE \textbf{Input:} Initial iterate $x^0$, initial stepsize $P_0$, candidate stepsize set $\mathcal{P}$, learning rate $\eta$, 
    and feedback $\ell_{x,\xi,\zeta} \in \{ r_{x,\xi,\zeta},\, h_{x,\xi,\zeta} \}$ defined by \eqref{eq:def-h}
    \FOR{$k = 0,1,2,\dots$}
        \STATE Sample $\xi^k$ and $\zeta^k$ uniformly and independently

        \STATE 
        $x^{k+1/2}=x^k - P_k \nabla f_{\xi^k}(x^k)$

        \STATE
        $x^{k+1} =
        \begin{cases}
        \arg\min_{x \in \{ x^k,\, x^{k+1/2} \}} f_{\zeta^k}(x), & \text{if use } h_{x,\xi,\zeta}, \\
        x^{k+1/2}, & \text{otherwise.}
        \end{cases}$

        \STATE 
        $P_{k+1}= \Pi_{\mathcal{P}} [P_k - \eta \nabla \ell_{x^k,\xi^k,\zeta^k}^k(P_k)]$
    \ENDFOR
\end{algorithmic}
\end{algorithm}

In this section, we establish convergence results for {\osgd}. Specifically, we consider two settings: 1) with access to the exact function value, and 2) with a noisy function value oracle (\cref{as:bound-f}). The first setting is justified by derivative-free optimization, where an exact function value oracle is available and the gradient is estimated by finite difference. 

We use the following assumption about the stochastic gradient and function value oracles throughout this section.

\begin{assumption}[Gradient norm condition]
  \label{as:bound-g}
  The stochastic gradient oracle $\nabla f_{\xi} (x)$ satisfies, $\mathbb{E}_{\xi} [\nabla f_{\xi} (x)]={\nabla}f (x)$, and for any iterate $x \in \{ x^k \}_{k = 1, 2, \ldots}$ and any $t > 0$,
    \begin{equation}\label{eq:hp-gradient-condition} 
    \mathbb{P}\!\left\{\,\|\nabla f_{\xi}(x) - \nabla f(x)\| \ge t \,\|\nabla f(x)\| \right\}
    \le
    2 \exp\!\big(-\tfrac{t^{2}}{2\sigma_1^{2}} \big).
    \end{equation}
    where $\sigma_1 > 0$ controls the relative noise level.  
    For a mini-batch of size $b$, the constant scales as $\sigma_1 / \sqrt{b}$.
\end{assumption}

\begin{remark}
  The gradient norm condition was first proposed in {\cite{carter1991global}}
  to analyze the stochastic trust region method, and is a mainstay of subsequent
  literature {\cite{berahas2021global,byrd2012sample}}. 
  In the derivative-free optimization setting, where stochastic gradient estimates are obtained by sampling, 
  the norm condition can be shown to hold with high probability {\cite{berahas2022theoretical}}.
\end{remark}

When the exact function value is not available, we use the following assumption on the function value oracle.

\begin{assumption}
  \label{as:bound-f}
  The stochastic function value oracle $f_{\xi} (x)$ satisfies,
  $\mathbb{E}_{\xi} [f_{\xi} (x)]$=$f (x)$, and for any iterate 
$x \in \{x^k\}_{k=1,2,\ldots}$ and any $t>0$,
\begin{equation}\label{eq:hp-function-oracle}
    \mathbb{P}\!\left\{
        |f_{\xi}(x) - f(x)| 
        \;\ge\; t \, |f(x) - f^\star|
    \right\}
    \;\le\;
    2\exp\!\left( -\tfrac{t^{2}}{2\sigma_0^{2}} \right),
\end{equation}
where $\sigma_0 > 0$ controls the relative noise level in the function value. 
We use $\sigma_0 = 0$ to denote the exact function value oracle. 
For a mini-batch of size $b$, the constant scales as $\sigma_0 / \sqrt{b}$.
\end{assumption}

We provide an example where \cref{as:bound-g,as:bound-f} hold. 
\begin{example}[Least-squares with interpolation and sub-Gaussian features]
\label{ex:gradient-norm}
Suppose
\[
    f(x) = \tfrac{1}{n} \textstyle \sum_{i=1}^n f_i(x),
    \qquad
    f_i(x) = \tfrac12 (a_i^\top x - b_i)^2,
\]
and assume an interpolation setting where $a_i^\top x^\star = b_i$ for all $i$, and $a_i$ are sub-Gaussian with parameter $\sigma_a$. Let $H \assign \tfrac1n \sum_{i=1}^n a_i a_i^\top$, then \cref{as:bound-g,as:bound-f} hold with $\sigma_0 = \Theta(\sigma_a^{2}/\lambda_{\min}(H))$ and $\sigma_1 = \Theta(\sigma_a^{2}/\lambda_{\min}(H))$. 
\end{example}

\subsection{Convergence results}

We have the following convergence guarantee for {\osgd} with feedback functions \eqref{eq:def-h}. \cref{tab:theoretical-results} summarizes the theoretical results in this section.

\begin{table}[ht]
    \centering
    \resizebox{\textwidth}{!}{\begin{tabular}{ccccc}
  \toprule
  Noise oracle & Function class & Feedback & Iteration complexity & Batchsize\\
  \midrule
  \multirow{3}{*}{\makecell{\cref{as:bound-g},\\ exact function value}}  & $L$-smooth, $\mu$-strongly convex & Ratio & $\tilde{\mathcal{O}} \big( \big( 1 +
  \sigma_1^2 \big) \kappa_{\star} \log
  \tfrac{1}{\varepsilon} \big)$ & $\tilde{\mathcal{O}} (1)$\\
  \  & $L$-smooth, $\mu$-strongly convex  & Hypergradient &
  $\tilde{\mathcal{O}} \big( (1 + \sigma_1^2) \kappa \log \tfrac{1}{\varepsilon}
  \big)$ & $\tilde{\mathcal{O}} (1)$\\
  \  & $L$-smooth, convex  & Hypergradient &
  $\tilde{\mathcal{O}} \big( \tfrac{L \Delta^2 (1 + \sigma_1^2)}{\varepsilon} \big)$
  & $\tilde{\mathcal{O}} (1)$\\
  \midrule
  \multirow{2}{*}{\makecell{\cref{as:bound-g,as:bound-f}}}  & $L$-smooth, $\mu$-strongly convex &  Ratio& $\tilde{\mathcal{O}} \big( \big( 1 +
  \sigma_1^2 \big) \kappa_{\star} \log
  \tfrac{1}{\varepsilon} \big)$ & $\tilde{\mathcal{O}} (\kappa_{\star}^2)$\\
  \  & $L$-smooth, $\mu$-strongly convex & Hypergradient &
  $\tilde{\mathcal{O}} \big( (1 + \sigma_1^2)^2 \kappa \log \tfrac{1}{\varepsilon}
  \big)$ & $\tilde{\mathcal{O}} (\kappa^2)$\\
  \bottomrule
\end{tabular}}
    \caption{Summary of Theoretical Results of {\osgd} (\cref{thm:conv-subG}).}
    \label{tab:theoretical-results}
\end{table}

\begin{theorem}[Convergence of {\osgd}]
  \label{thm:conv-subG}
  Under \cref{as:bound-g,as:bound-f}, suppose we run {\osgd} for $K$
  iterations. 
  For any $\delta \in (0, 1)$, define $\gamma
  (K, \delta) = 4 \sqrt{\max \left\{ \log \tfrac{4 K}{\delta}, 1 \right\}}$ and 
  let $\tilde{\mathcal{O}}$ hides polynomial logarithmic terms of $K$ and $1/{\delta}$. 
  Then with probability $\geq 1 - \delta$ we have the following convergence results.
  \begin{enumerate}
    \item {\emph{Ratio feedback, strongly convex.}} Suppose $\sigma_1 < \tfrac{1}{2 \gamma (K, \delta)}$ and $\sigma_0 =\tilde{\mathcal{O}}\large( \tfrac{1}{L^2 D^2 \kappa_{\star}} \large)$.
    
    \[ \tfrac{f (x^{K}) - f (x^{\star})}{f (x^0) - f (x^{\star})} \leq
       \big( 1 - \tfrac{1}{2 \left( 1 + \gamma
       (K, \delta)^2 \sigma_1^2  \right) \kappa_{\star}} +
       \tilde{\mathcal{O}} \big( \tfrac{1}{\sqrt{K}} \big) \big)^K . \]
    \item {\emph{Hypergradient feedback, strongly convex.}} Suppose $\sigma_1 =\tilde{\mathcal{O}} \left( \tfrac{1}{L^2 D^2} \right)$ and
    $\sigma_0 =\tilde{\mathcal{O}} \left( \tfrac{1}{L^2 D^2 \kappa} \right)$.
    \[ \tfrac{f (x^{K}) - f (x^{\star})}{f (x^0) - f (x^{\star})} \leq
       \big( 1 - \tfrac{1}{2 (1 + \sigma_1^2 ) \gamma (K, \delta)^2 \kappa} +
       \tilde{\mathcal{O}} \big( \tfrac{1}{\sqrt{K}} \big) \big)^K . \]
    \item {\emph{Hypergradient feedback, convex.}} Suppose $\sigma_1 =\tilde{\mathcal{O}} \left( \tfrac{1}{L^2 D^2} \right)$ and $\sigma_0 = 0$.
    \[ f (x^{K}) - f (x^{\star}) \leq \min \left\{ \tfrac{\Delta^2}{\max
       \big\{ K \big( \frac{1}{4 L (1 + \gamma (K, \delta)^2 \sigma_1^2)} -
       \tilde{\mathcal{O}} \big( \tfrac{1}{\sqrt{K}} \big) \big), 0 \big\}}, f (x^0) - f^{\star}
       \right\} . \]
  \end{enumerate}
\end{theorem}

\cref{thm:conv-subG} suggests that when the stochastic noise is small, the asymptotic convergence of {\osgd} is similar to the deterministic setting. 

The requirement on the noise level $\sigma_0, \sigma_1$ can always be satisfied by using a sufficiently large batchsize. We distinguish the following two settings, when $\sigma_0=0$ and $\sigma_0>0$. 
If the function value oracle is exact, $\sigma_0=0$, 
\cref{thm:conv-subG} guarantees linear convergence as long as $\sigma_1 = \mathcal{O}(\tfrac{1}{L^2 D^2})$ for both ratio and hypergradient feedback. 
Therefore, the batchsize needed is $\tilde{\mathcal{O}}(1)$, as $D = \mathcal{O}(1/L)$ by \cref{as:diam}. 
On the other hand, for an approximate function value oracle $\sigma_0>0$,
the theorem guarantees linear convergence for 
ratio feedback with a batchsize $\tilde{\mathcal{O}}(\kappa_\star^2)$, and for hypergradient feedback with a batchsize $\tilde{\mathcal{O}}(\kappa^2)$. These results are summarized in \cref{tab:theoretical-results}.

Large batchsizes are increasingly practical given modern computational constraints.
It is typical to select a batchsize based on the number of processors available, 
as the time required is the same as needed for a smaller batchsize.
Hence modern hardware accelerators like GPUs, with hundreds or thousands of parallel processors, reward algorithms that can make efficient use of large batchsizes.

We note that access to an exact function value oracle improves several aspects of the convergence theory. First, it reduces the batchsize requirement from $\tilde{\mathcal{O}}(\kappa^2)$ to $\tilde{\mathcal{O}}(1)$, as 1) with exact function values, the ratio feedback can be evaluated exactly, and 2) the stochastic online gradient $\tfrac{\nabla f_{\zeta}(x - P \nabla f_{\xi} (x)) \nabla f(x)^\top}{f (x) - f^{\star}}$ is an unbiased gradient estimator of the deterministic ratio feedback $\tfrac{f(x - P \nabla f_{\xi} (x)) - f^\star}{f (x) - f^{\star}}$. Moreover, an exact function value oracle improves the regret bound for ratio feedback from linear to sublinear, as shown by \cref{lem:regret-exact-func}. Finally, for the hypergradient feedback, an exact function value oracle makes it possible to ensure descent and prevent divergence.

We can combine these considerations to understand when {\osgd} offers an improved sample complexity compared to {\sgd}. With an exact function value oracle, 
the sample complexity of {\osgd} with ratio feedback is $\tilde{\mathcal{O}}((1+\sigma_1^2)\kappa_\star \log \tfrac{1}{\epsilon})$, which improves the result of {\sgd} from $\kappa$ to a much smaller number $\kappa_\star$. The sample complexity of hypergradient feedback matches the result of {\sgd}. Without an exact function value oracle, the sample complexity of ratio feedback is $\tilde{\mathcal{O}}((1+\sigma_1^2)\kappa_\star^3 \log \tfrac{1}{\epsilon})$, which improves on {\sgd} when $\kappa_\star^3<\kappa$. The sample complexity of hypergradient feedback is $\tilde{\mathcal{O}}((1+\sigma_1^2)\kappa^3 \log \tfrac{1}{\epsilon})$, which is worse than the result of {\sgd}. However, our empirical experiments reveal that {\osgd} outperforms {\sgd} and its variants with the same batchsize.

\section{{\sosgm} with variance reduction}
\label{sec:vr}

This section applies {\os} at the outer-loop level of {\svrg} to learn and adapt the matrix stepsize across epochs.
Because the full gradient is available at each outer iterate, the feedback for stepsize selection is deterministic, and convergence analysis does not require the noise oracle assumed in \cref{sec:large-batchsize}.

We begin by recalling the {\svrg} method and a practically effective extension that uses heavy-ball momentum (\cref{alg:osgm-svrg} with {\os} learning rate $\eta=0$ and constant momentum $\beta_k=\beta$). {\svrg} uses a double-loop structure: at the start of outer epoch $k$,
it computes the full gradient at a snapshot point $\tilde{x}^k$,
and then performs $m$ inner iterations using VR estimators.
We also allow a heavy-ball momentum term parameterized by $\beta$ inside the inner loop.
When $\beta=0$, the algorithm is {\svrg}. For $\beta>0$, we will call the algorithm {\sh}.

{\svrg} with acceleration has appeared in the literature \cite{nitanda2014stochastic,lin2015universal,allen2018katyusha,shang2017fast}, but not with heavy-ball momentum.
Experimentally, Nesterov acceleration and heavy-ball momentum perform about equally well,
but heavy-ball momentum allows for theoretical guarantees in the context of {\os} that
are currently unknown for Nesterov momentum \cite{chu2025gradient}.

\subsection{\texttt{OSGM-SVRG}}

In this section, we develop a unified framework {\ov} that uses {\os} at each outer iteration of {\svrg}
to learn and improve the {\svrg} stepsize. The {\ov} framework is presented in \cref{alg:osgm-svrg}.
Concretely, at outer iteration $k$, given the deterministic gradient $\nabla f(\tilde{x}^k)$ at the snapshot iterate, 
{\ov} uses either the ratio or hypergradient feedback with regularization $\rho>0$,
\begin{equation}\label{eq:ov-feedback}
  r^{\rho}_{\tilde{x}^k} (P) = r_{\tilde{x}^k} (P) + \tfrac{\rho}{2} \| P \|_F^2\quad \text{and} \quad h^{\rho}_{\tilde{x}^k} (P) = h_{\tilde{x}^k} (P) + \tfrac{\rho}{2} \| P \|_F^2,
\end{equation}
where $r_{\tilde{x}^k}$ and $h_{\tilde{x}^k}$ are defined in \eqref{eq:dtm-feedback}. 
The regularization delivers an implicit bound on the size of the stepsize (\cref{prop:vr-stepsize-bound})
that will be important for the theoretical guarantees that follow.
With this feedback design, apply online gradient descent to update the stepsize $P_k$,
\begin{equation}\label{eq:svrg-P-update}
  P_k = \Pi_{\mathcal{P}}
        [
            (1 - \eta\rho)P_{k-1}
            + \eta \nabla \ell_{\tilde{x}^k} (P_{k-1})
        ], \quad \text{where} \quad \ell_{\tilde{x}^k} \in \{ r_{\tilde{x}^k}, h_{\tilde{x}^k} \}.
\end{equation}
In the inner-loop of {\svrg} (Line 9 of \cref{alg:osgm-svrg}), 
{\ov} moderates the learned stepsize with a decay factor $c \leq 1$ to ensure convergence in the context of stochastic gradient updates. 
We discuss the choice of decay factor further in \cref{thm:osgm-svrg}. 

\begin{algorithm}[ht!]
\caption{\texttt{OSGM-SVRG}}
\label{alg:osgm-svrg}
\begin{algorithmic}[1]
\STATE \textbf{Input:} Initial iterate $\tilde{x}^0$, and stepsize $P_0$, epoch length $m$, {\os} learning rate $\eta$, decay factor $c$, candidate stepsize set $\mathcal{P}$, momentum sequence $\{\beta_k\}$, regularization $\rho$, and deterministic feedback function $\ell_{\tilde{x}^k}^{\rho} \in \{ r_{\tilde{x}^k}^{\rho}, h_{\tilde{x}^k}^{\rho} \}$ defined by \eqref{eq:ov-feedback}.
\FOR{$k=0,1,2,\dots$}
    \STATE Compute snapshot gradient $\nabla f(\tilde{x}^k) = \tfrac1n\sum_{i=1}^n \nabla f_i(\tilde{x}^k)$
    
    \STATE Set $x^0 = \tilde{x}^k$
    \FOR{$t=0,\dots,m-1$}
        \STATE Sample $\xi^t$ uniformly
        \STATE
        $
            g_t = \nabla f_{\xi^t}(x^t) - \nabla f_{\xi^t}(\tilde{x}^k) + \nabla f(\tilde{x}^k)
        $
        \STATE 
        $
            x^{t+1}
            =
            x^t
            - c P_k g_t
            + \beta_k (x^t - x^{t-1})
        $
    \ENDFOR
    \STATE
    $
    P_{k+1} = \Pi_{\mathcal{P}}
        [
            P_{k}
            - \eta \nabla \ell_{\tilde{x}^{k}}^{\rho} (P_{k})
        ].
    $
    \STATE Choose $\tilde{x}^{k+1}$ uniformly from $\{x^0,\dots,x^m\}$
\ENDFOR
\end{algorithmic}
\end{algorithm}

\begin{remark}
  {\ov} is a general framework and can be reduced to {\svrg} and {\sh} by setting learning rate $\eta=0$, constant momentum $\beta_k=\beta$. The choice of regularization $\rho$ is justified by \cref{prop:vr-stepsize-bound} for scalar and matrix stepsize.
\end{remark}

\subsection{Convergence analysis}

In this section, we establish a generic convergence result for {\ov}. First, we bound the potential decrease on expectation in each inner loop update. Within epoch $k$, define the filtration $\mathcal{F}_t$, the $\sigma$-algebra generated by all randomness up to and including step $t$. We write
\(
    \mathbb{E}_t[\cdot]
    :=
    \mathbb{E}\!\left[\,\cdot \,\middle|\, \mathcal{F}_t \right]
\)
for the conditional expectation with respect to this filtration.

\begin{lemma}[Potential reduction]
  \label{lem:vr-pot}
  For any epoch $k$ and any inner step $t$, conditional on the filtration $\mathcal{F}_t$, the expected potential satisfies the following bounds:
  \begin{enumerate}
    \item \emph{No momentum.} Suppose $\beta_k = 0$ and assume stepsize
    $\underline{\alpha} I \preceq P_k \preceq \bar{\alpha} I$. Define $V^t \assign f (x^t) - f^{\star}$. Then
    \begin{equation}
      \mathbb{E}_{t} [V^{t + 1}] \leq V^t - \left( 2 c \underline{\alpha} \mu - 2 c^2 \bar{\alpha}^2L^2 \right) (f (x^t) - f^{\star}) + 2 c^2 \bar{\alpha}^2 L^2 (f
    (\tilde{x}^k) - f^{\star}). \label{pot-svrg}
    \end{equation}
    For a scalar stepsize $P_k = \alpha I$, the upper and lower bounds $\underline{\alpha} = \bar{\alpha} = \alpha$ match.
    
    \item \emph{Bounded momentum.} Suppose $0 < \beta_k \leq  \bar{\beta}$ and assume a scalar stepsize $P_k = \alpha I$. Define $V^t \assign f(x^t) - f^{\star} + \tfrac{L}{2} \| x^t - x^{t - 1} \|^2$. Then
    \begin{align}
      \begin{split}
        \mathbb{E}_{t} [V^{t + 1}]  \leq{} & V^t - \left( \tfrac{1}{2} - \bar{\beta}^2 \right) L \| x^t - x^{t-1} \|^2 - (c\alpha - 2 c^2 \alpha^2 L \kappa) \| \nabla f (x^t) \|^2\\
      &+ (1 - 2 c \alpha L) \bar{\beta} \langle \nabla f (x^t), x^t - x^{t-1} \rangle + 4
      c^2 \alpha^2 L^2 (f (\tilde{x}^k) - f^{\star}).
      \end{split}
      \label{pot-hbm}
    \end{align}
  \end{enumerate}
\end{lemma}

\begin{remark}
  The potential function defined in the bounded momentum setting (ii) is inspired by {\cite{kulakova2018non}}, which introduces potential function $f (x^t) - f^{\star} + \tfrac{1 - \alpha L}{2 \alpha} \| x^t - x^{t - 1} \|^2$ to analyze the convergence of deterministic {\hbm}. Given the bound on the potential reduction \eqref{pot-hbm}, we can choose the stepsize $\alpha$ and the upper bound of momentum $\bar \beta$ appropriately to guarantee a strict decrease of the potential.
\end{remark}

From \cref{lem:vr-pot}, we can directly derive the convergence of {\svrg} and {\sh}.

\begin{proposition}[Convergence of {\svrg} and {\sh}]
  \label{prop:svrg-conv} Let decay factor $c=1$. Use constant scalar stepsize $P_k=\alpha I$ and constant momentum $\beta_k=\beta$ in \cref{alg:osgm-svrg}.
  \begin{enumerate}
    \item \emph{{\svrg}.} Suppose momentum $\beta = 0$, and stepsize $\alpha$ satisfies $\alpha <
    \tfrac{1}{\kappa L}$. Then
    \[ \mathbb{E} [f (\tilde{x}^{k}) - f^{\star}] \leq \left(\tfrac{1 + 2 m \alpha^2 L^2}{2 m (\alpha \mu - \alpha^2 L^2)}\right)^k (f(\tilde{x}^0) - f^{\star}). \]
    \item \emph{{\sh}.} Suppose momentum $\bar{\beta} \leq \sqrt{\alpha L - 2 \alpha^2
  L^2 \kappa}$, and stepsize $\alpha
  < \tfrac{1}{2 \kappa L}$. Then
    \[ \mathbb{E} [f (\tilde{x}^{k}) - f^{\star}] \leq \left( \tfrac{1 + 4 m \alpha^2 L^2}{m( 2 \alpha \mu - 4
    \alpha^2 L^2 - {2 \bar{\beta}^2}/{\kappa} )} \right)^k (f(\tilde{x}^0) - f^{\star}). \]
    If $\alpha = \tfrac{1}{8 \kappa
  L}$ and $\bar{\beta} = \tfrac{1}{\sqrt{32 \kappa}}$, the contraction ratio
  is $\tfrac{8 \kappa^2}{m} + \tfrac{1}{2}$.
  \end{enumerate}
\end{proposition}

\begin{remark}
  From \cref{prop:svrg-conv}, to derive linear convergence, the epoch length of {\svrg} is $m = \mathcal{O}(\kappa^2)$, and the epoch length for {\sh} is $m = \mathcal{O}(\kappa^2)$. Our result has a inferior dependence on $\kappa$ compared to the result of {\cite{johnson2013accelerating}}, which requires $m = \mathcal{O} (\kappa)$. This difference results from the choice of potential function: {\cite{johnson2013accelerating}} uses the potential $\| x^t - x^{\star}\|^2$.
  However, this potential cannot be directly extended to the matrix stepsize case where $\underline{\alpha} I \preceq P_k \preceq \bar{\alpha} I$. 
  Therefore, in the analysis in this section, we use function value gap as the potential.
\end{remark}

We now develop a convergence analysis for {\ov}.
First, we show that the stepsize $P_k$ is bounded. 
Then we show {\ov} converges linearly as long as the stepsize $P$ is bounded,
which we can ensure using stepsize decay ($c < 1$ in \cref{alg:osgm-svrg}) or projection (bounded $\mathcal P$) as a safeguard. This analysis approach is also used in the analysis of the stochastic \texttt{L-BFGS} method {\cite{moritz2016linearly}} and Barzilai-Borwein stepsize {\cite{tan2016barzilai}} in the {\svrg} framework.

\begin{proposition}[Bounded stepsize]
  \label{prop:vr-stepsize-bound}Consider the stepsize $P_k$ updated by {\ov}.
  \begin{enumerate}
    \item \emph{Scalar Stepsize.} Suppose $\mathcal{P} = \{ \alpha I : \alpha \in \mathbb{R} \}$. Let $P_k = \alpha_k I$, regularization $\rho = 0$. Assume initial stepsize $ \alpha_0 \in \left[\tfrac{1}{L}, \tfrac{1}{\mu} \right]$, and {\os} learning rate satisfies $\eta \leq \tfrac{1}{L}$ for hypergradient feedback, $\eta \leq \tfrac{1}{2L^2}$ for ratio feedback. Then
    \[ \alpha_k \in [ \tfrac{1}{L}, \tfrac{1}{\mu} ], \quad \forall k. \]
  
    \item \emph{Matrix Stepsize.} Suppose $\mathcal{P} =\mathbb{S}^{n}$ or $\mathcal{P} = \{\operatorname{diag}(d):d \in \mathbb{R}^n\}$. Let regularization $\rho = 2 L$, learning rate $\eta \leq \tfrac{1}{L}$ for hypergradient feedback, $\rho = 4 L^2$, $\eta \leq \tfrac{1}{2L^2}$ for ratio feedback. Assume initial stepsize $P_0 = 0$. Then
    \[ P_k \preceq \tfrac{1}{L} I, \quad \forall k. \]
  \end{enumerate}
\end{proposition}

\begin{remark}
\cref{prop:vr-stepsize-bound} highlights the self-adaptivity of
{\ov}. For scalar stepsize, even \emph{without} regularization, the update rule automatically constrains the stepsize within a bounded interval. For matrix stepsizes, regularization is needed to keep $\|P_k\|_2$ upper-bounded.  
\end{remark}

With \cref{lem:vr-pot} and \cref{prop:vr-stepsize-bound}, we are ready to get the main convergence result.

\begin{theorem}[Convergence of {\ov}]
  \label{thm:osgm-svrg}
  Consider the following instantiations,
  \begin{enumerate}
    \item \emph{No momentum, scalar stepsize.} Suppose $\beta_k=0$, $\mathcal{P} = \{ \alpha I : \alpha \in \mathbb{R} \}$, and $c \leq \tfrac{1}{\kappa (\kappa+ 1)}$. 
    Under the conditions in \cref{prop:vr-stepsize-bound} (i), then
    \[ \mathbb{E} [f (\tilde{x}^k) - f^{\star}] 
    \leq 
    \left( 
      \tfrac{\kappa}{2 m (c - c^2 \kappa)} + \tfrac{c\kappa^2}{1 - c \kappa^2}
    \right)^k 
    [f (\tilde{x}^0) - f^{\star}] . \]

    \item \emph{No momentum, matrix stepsize.} Suppose $\beta_k=0$, $\mathcal{P} = \{ P \in \mathbb{S}^{n}: P \succeq \underline{\alpha} I\}$ or $\mathcal{P} = \{P=\operatorname{diag}(d):d \in \mathbb{R}^n, d \geq \underline{\alpha}\}$, and $c \leq \underline{\alpha} \mu$. 
    Under the conditions in \cref{prop:vr-stepsize-bound} (ii), then
    \[ \mathbb{E} [f (\tilde{x}^k) - f^{\star}] \leq \left( 
      \tfrac{1}{2 m \left( c \underline{\alpha} \mu - c^2
     \right)} + \tfrac{c}{\left( \underline{\alpha} \mu - c \right)}
      \right)^k [f(\tilde{x}^0) - f^{\star}] . \]

    \item \emph{Bounded momentum, scalar stepsize.} Suppose $0 \leq \beta_k \leq \tfrac{1}{\kappa^2}$, $\mathcal{P} = \{ \alpha I : \alpha \in
    \mathbb{R} \}$, $c \leq \tfrac{1}{4 \kappa (\kappa + 1)}$,
    and ${\beta_k} \leq \sqrt{\tfrac{c}{2}}$. 
    Under the conditions in \cref{prop:vr-stepsize-bound} (i), then
    \[ \mathbb{E} [f (\tilde{x}^k) - f^{\star}] 
      \leq 
      \left( 
        \tfrac{\kappa}{m (c - c^2 \kappa)} + \tfrac{4 c \kappa^2}{1 - 4 c \kappa^2}
      \right)^k 
       [f (\tilde{x}^0) - f^{\star}]. \]
  \end{enumerate}
\end{theorem}

\begin{remark}
  Using \cref{thm:osgm-svrg}, we can choose decay factor $c$ and epoch length $m$ to guarantee linear convergence with contraction rate $\tfrac{3}{4}$.
  For setting (i), a feasible choice is $c = \tfrac{1}{3 \kappa^2}$ and $m =
  9 \kappa^3$.
  For setting (ii), a feasible choice is $c = \tfrac{\underline{\alpha} \mu}{3}$
  and $m = \tfrac{9}{\underline{\alpha}^2 \mu^2}$.
  For setting (iii), a feasible choice is $c = \tfrac{1}{12 \kappa^2}$ and $m = 72 \kappa^3$. 
  This result has inferior dependence on $\kappa$ compared to the results of {\svrg} and {\sh} (\cref{prop:svrg-conv}). 
  However, our method rewards practical faster convergence.
\end{remark}

\section{Experiments}\label{sec:exp}

The previous sections introduce {\osgd} and {\ov}. 
This section benchmarks the performance of these algorithms and some more practical variants on machine learning and deep learning tasks. Since $f^\star$ is typically unknown, 
all experiments use hypergradient feedback.

\subsection{Practical variants}

The algorithms {\osgd} and {\ov} are designed to admit clean convergence proofs.
This section introduces variants of these algorithms optimized for performance 
rather than theoretical guarantees, which is summarized in \cref{tab:osgm-variants}. 
As shown in \cref{counterexample}, the practical variants can fail to converge in adversarial settings not covered by our theory.
Nevertheless, they uniformly outperform the theoretical variants on the benchmarks we consider, suggesting that the counterexample conditions are rarely encountered in practice. 
As a result, numerical results in the main paper show only results for these practical variants. 
Results for the original, provably convergent variants, appear in Appendix~\ref{app:exp-fig}.

\begin{table}[ht]
  \centering
  \resizebox{\textwidth}{!}{\begin{tabular}{cccc}
      \toprule
      \textbf{Family} & \textbf{Algorithm} & \textbf{Iteration} & \textbf{Feedback}\\
      \midrule
      \multirow{2}{*}{\makecell{{\sgd}\\Variants}}
      & {{\osgd} (\cref{alg:sosgm})}
      & $x^{k+1} = x^k - \textcolor{red}{P_k}\nabla f_{\xi^k}(x^k)$
      & $ 
        \tfrac{
          f_{\zeta^k}\!(x^{k+1}(\textcolor{red}{P_k})) - f_{\zeta^k}(x^k)
        }{
          \|\nabla f_{\xi^k}(x^k)\|^2
        }$\\
      \cmidrule(lr){2-4}
      & {{\osgdh} (\cref{alg:prac-osgd})}
      & $x^{k+1} = x^k - \textcolor{red}{P_k}\nabla f_{\xi^k}(x^k)$
      & $ 
        \tfrac{
          f_{\xi^k}\!(x^{k+1}(\textcolor{red}{P_k})) - f_{\xi^k}(x^k)
        }{
          \|\nabla f_{\xi^k}(x^k)\|^2
        }$\\
        \midrule
      \multirow{3}{*}{\makecell{Variance\\Reduction}}
      & {{\ov} (\cref{alg:osgm-svrg})}
      & $x^{t+1} = x^t - c\,\textcolor{red}{P_k}\, g_t + \beta_t\, m^t$
      & $ 
        \tfrac{
          f\!(\tilde{x}^k - \textcolor{red}{P_k}\nabla f(\tilde{x}^k)) - f(\tilde{x}^k)
        }{
          \|\nabla f(\tilde{x}^k)\|^2
        }+\tfrac{\rho}{2} \|\textcolor{red}{P_k}\|_F^2$\\
      \cmidrule(lr){2-4}
      & {\ovh (\cref{alg:prac-osgm-vr})}
      & $x^{t+1} = x^t - \textcolor{red}{P_t}\, g_t + \textcolor{red}{\beta_t}\, m^t$
      & $ 
        \tfrac{
          f_{\xi^t}\!(x^{t+1}(\textcolor{red}{P_t},\textcolor{red}{\beta_t})) - f_{\xi^t}(x^t)
        }{
          \|g_t\|^2 + \|m^t\|^2
        }$\\
      \cmidrule(lr){2-4}
      & {\osk (\cref{alg:osgm-sketchy})}
      & $x^{t+1} = x^t - \textcolor{red}{\alpha_t}\bar{P}_k\, g_t - \textcolor{red}{D_t} g_t + \textcolor{red}{\beta_t}\, m^t$
      & $ 
        \tfrac{
          f_{\xi^t}\!(x^{t+1}(\textcolor{red}{\alpha_t},\textcolor{red}{D_t},\textcolor{red}{\beta_t})) - f_{\xi^t}(x^t)
        }{
          \|g_t\|^2 + \|m^t\|^2
        }$\\
      \bottomrule
    \end{tabular}}
  \caption{
    Summary of {\sosgm} algorithms. 
    The parameters tuned by {\os} are marked as {\textcolor{red}{red}}. 
    Methods with provable convergence guarantees are marked in colors. 
    For {\osk}, $\bar{P}_k$ is the random low-rank preconditioner updated by sketchy methods per epoch. 
    We use $x^{k+1}(\textcolor{red}{P_k})$ to denote parameterized update rule. 
    For VR methods, $g_t \assign \nabla f_{\xi^t}(x^k)-\nabla f_{\xi^t}(\tilde{x}^k)+\nabla f(\tilde{x}^k)$ and $m^t \assign x^t - x^{t-1}$.
  }
  \label{tab:osgm-variants}
\end{table}

\paragraph{Practical {\osgd} variant and choice of feedback} In \cref{sec:large-batchsize}, we discussed the choice of feedback function from a theoretical perspective. We saw that computing the gradient and evaluating the feedback on a different sample from the data distribution was necessary to guarantee convergence. 
However, in practice, this feedback can result in a conservative stepsize choice and slow convergence.
Conversely, although feedback \eqref{eq:hoverfit}, which uses the same sample to compute the gradient and evaluate the feedback, may not converge in the worst case, it converges quickly in practice. 
Therefore, our numerical results in this section use variant {\osgdh} powered by feedback \eqref{eq:hoverfit}. An empirical comparison between feedbacks \eqref{eq:hoverfit} and \eqref{eq:def-h} appears in Appendix~\ref{app:exp-fig2}.

\paragraph{Practical {\ov} variants} {\ov} applies {\os} as the outer-loop stepsize scheduler. In practice, we can apply {\os} in each inner loop and tune the stepsize and momentum simultaneously (as discussed in \cite{chu2025gradient}). We use a diagonal stepsize since it is more efficient in memory and compute than a matrix stepsize. The pseudocode for the practical variant {\ov} appears as \cref{alg:prac-osgm-vr} in Appendix~\ref{app:exp-algs}. 

{\os} can be implemented on top of {\promise} methods \cite{frangella2024promise}, 
to tune the stepsize of a low rank preconditioner,
or in the context of an optimizer with heavy-ball momentum, to tune the momentum coefficient \cite{chu2025provable}. 
Our strongest algorithm in practice, {\osk}, uses {\os} to tune both the diagonal stepsize and momentum parameter of a heavy-ball variant of {\ssvrg}.
Pseudocode is presented as \cref{alg:osgm-sketchy}.

\begin{algorithm}[ht!]
\caption{{\osk}}
\label{alg:osgm-sketchy}
\begin{algorithmic}[1]
\STATE \textbf{Input:} Initial $\tilde{x}^0$, $P_0=0$, $\beta_0=0$, epoch length $m$, {\os} learning rates $\eta_P$ and $\eta_{\beta}$, decay factor $c$, regularization $\rho$, candidate set of diagonal stepsize and momentum $\mathcal{P}$, $\mathcal{B}$
\FOR{$k=0,1,2,\dots$}
    \STATE Compute snapshot gradient $\nabla f(\tilde{x}^k) = \tfrac1n\sum_{i=1}^n \nabla f_i(\tilde{x}^k)$
    \STATE Estimate low-rank preconditioner $\bar P_k$ by sketchy methods
    
    \STATE Set $x^0 = \tilde{x}^k$
    \FOR{$t=0,\dots,m-1$}
        \STATE Sample $\xi^t$ uniformly
        \STATE
        $
            g_t = \nabla f_{\xi^t}(x^t) - \nabla f_{\xi^t}(\tilde{x}^k) + \nabla f(\tilde{x}^k)
        $
        \STATE
        $x^{t+1} = x^t
        - \alpha_{t} \bar{P}_k g_t
        - D_{t} g_t
        + \beta_{t}(x^t - x^{t-1})$

        \STATE Define feedback $\ell_{t}(\alpha, D, \beta)=\tfrac{f(x^t- \alpha \bar{P}_k g_t - D g_t+ \beta(x^t - x^{t-1})) - f(x^t)}{\|g_t\|^2 + \|x^t - x^{t-1}\|^2}$

        \STATE Update scalar stepsize: \quad \, $\alpha_{t+1} = \alpha_{t} - \eta_P \nabla_{\alpha} \ell_{t}(\alpha_t, D_t, \beta_t)$

        \STATE Update diagonal stepsize: $D_{t+1} = \Pi_{\mathcal{P}}\big[D_{t} - \eta_P \nabla_D \ell_{t}(\alpha_t, D_t, \beta_t)\big]$

        \STATE Update momentum:\qquad \quad $\beta_{t+1} = \Pi_{\mathcal{B}} \big[ \beta_{t} - \eta_{\beta} \nabla_{\beta} \ell_{t}(\alpha_t, D_t, \beta_t) \big]$
    \ENDFOR
    \STATE Choose $\tilde{x}^{k+1}$ uniformly from $\{x^0,\dots,x^m\}$, set $\alpha_0=\alpha_m, D_0=D_m, \beta_0=\beta_m$
\ENDFOR
\end{algorithmic}
\end{algorithm}

VR methods offer linear convergence and perform best for statistical learning applications. 
In contrast, for non-convex problems such as training deep neural networks, VR methods tend to underperform. Hence our experiments showcase the  methods on statistical learning, and  {\sgd} variants for deep learning.

\subsection{Statistical learning}

We benchmark {\ovh} and {\osk} on logistic regression with L2 regularization and ridge regression problems. 
We use datasets from LIBSVM \cite{chang2011libsvm} and OpenML \cite{vanschoren2013openml}, and set the batchsize to 256. 
The regularization parameter of logistic and ridge regression is ${10^{-2}}/{n}$.
We present more details on the datasets in Appendix~\ref{app:exp-dataset}.

\paragraph{Benchmark algorithms} We benchmark the following variance reduction algorithms.
\begin{itemize}
  \item \emph{Baseline VR optimizers:} {\svrg} \cite{johnson2013accelerating}, {\saga} \cite{defazio2014saga}, and {\kat} (\texttt{Loopless Katyusha} \cite{kovalev2020don}) with tuned stepsize. For {\svrg}, the update frequency is $m=\lceil n/256 \rceil$.
  \item \emph{{\promise} suite} \cite{frangella2024promise}. {\ssvrg}, {\ssaga}, and {\skat} with Nystr\"om Subsampled Newton preconditioner, rank 10, and default stepsize.
  \item \emph{Practical {\ov} variants.} {\ovh} and {\osk} with default {\os} learning rate for stepsize $\eta_P=1/L$, and default learning rate for momentum $\eta_{\beta}=0.1$.
\end{itemize}
We do not show performance of {\sgd} or {\osgdh} in our experiments because they perform much worse on benchmark tasks, as they converge sublinearly.
We use default stepsize for {\promise} suite since it already significantly outperforms the tuned baseline optimizers as shown by the experiments in \cite{frangella2024promise}.

\paragraph{Suboptimality experiments} \cref{fig:subopt} shows performance plots for logistic regression and ridge regression. The $y$-axis represents the suboptimality $f(x) - f^\star$. 

All these methods converge linearly. Yet a quick examination of the figures shows that as a practical matter, the baseline  optimizers do \emph{not} converge to a high-accuracy solution even after hundreds of epochs. In contrast, our strongest method {\osk} reaches high-accuracy regimes ($10^{-6}$ -- $10^{-12}$) normally considered beyond the reach of stochastic optimizers.

We observe that {\ovh} always outperforms the baseline optimizers ({\svrg}, {\saga}, {\kat}). {\ovh} is competitive with {\promise} suite. However, {\ovh} has lower memory and per-iteration compute cost, as methods of the {\promise} suite require storing a $d \times r$ matrix preconditioner (where $r$ is the rank), while {\ovh} stores only a $d$-dimensional vector to represent a diagonal preconditioner.

\begin{figure}[ht]
  \centering
  \includegraphics[width=1.\linewidth]{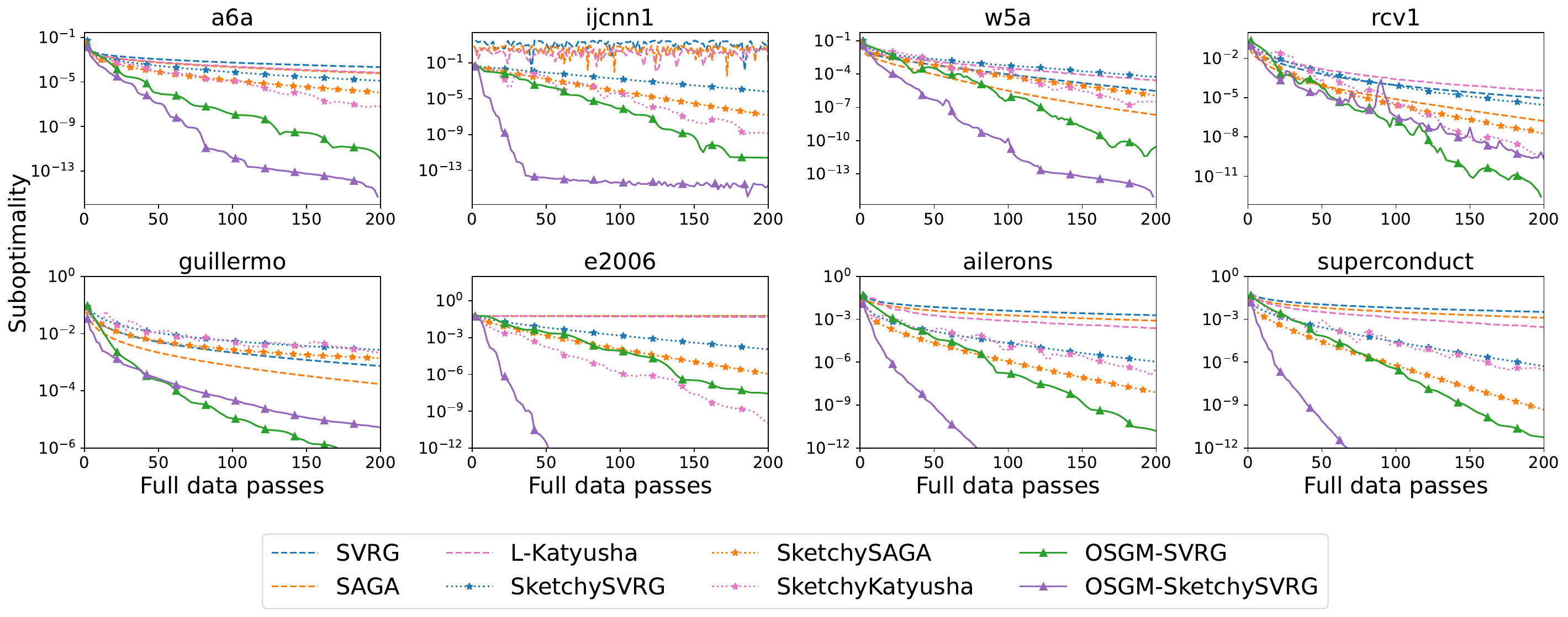}
  \caption{Suboptimality plots. First row: logistic regression. Second row: ridge regression.}
  \label{fig:subopt}
\end{figure}

Following the discussion in \cite{allen2018katyusha}, we emphasize the practical importance of high-accuracy solutions (e.g., function value gap $\leq 10^{-7}$). In particular, applications that use multiple black-box calls to ERM solvers \cite{allen2014optimization,frostig2016principal} can accumulate errors across calls, which makes high-accuracy solvers essential. 

\paragraph{Performance experiments} We compare the benchmark algorithms on a testbed of 47 medium-sized problems, including 31 logistic and 16 ridge regression problems. The primary metrics are the wall-clock time and the number of full data passes to reach suboptimality within $10^{-4}$ of the minimum. We set the budget as 600 seconds and 200 data passes. 

\cref{fig:performance} shows the performance plots. On logistic regression, {\osk} dominates the benchmark algorithms on both metrics and solves all the instances. {\ovh} is comparable with the best {\promise} variant. On ridge regression, both {\osk} and {\ovh} tie for solving the most instances within the time budget. {\osk} offers lower iteration counts, while {\ovh} delivers the fastest solve times.

\begin{figure}[ht]
  \centering
  \includegraphics[width=1.\linewidth]{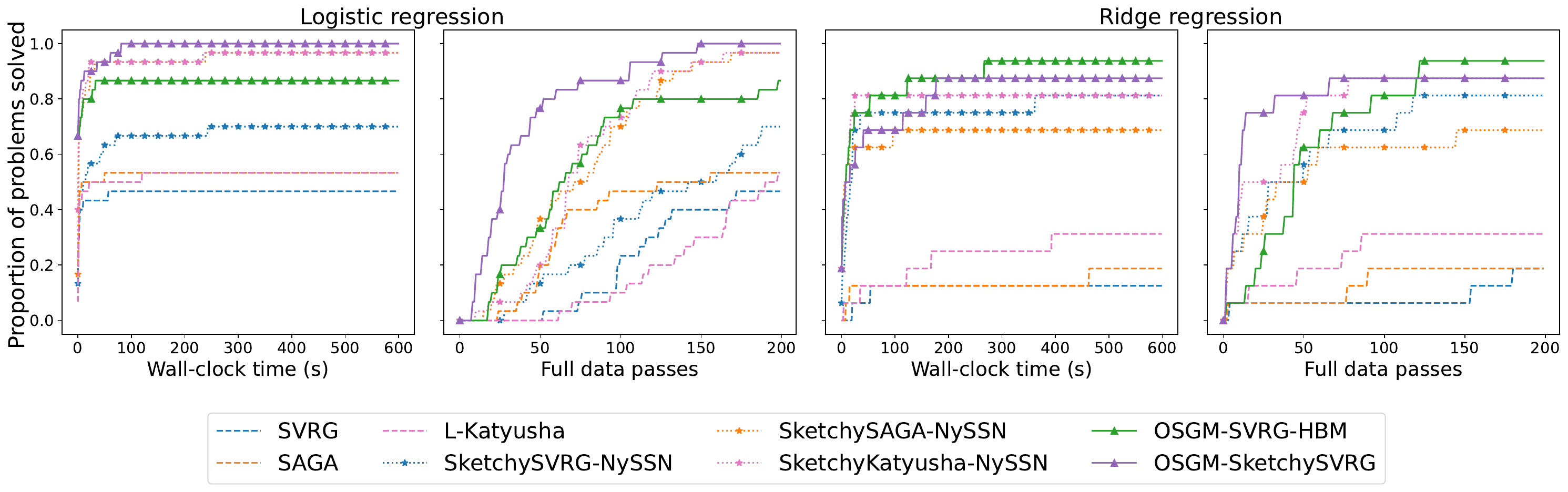}
  \caption{Performance plots. Left two: logistic regression. Right two: ridge regression.}
  \label{fig:performance}
\end{figure}

\subsection{Deep learning}

We benchmark {\osgdh} on training neural networks. 

\paragraph{Benchmark problems} We benchmark on the following problems, as in \cite{baydin2018hypergradient}: 1) Train an MLP model with two fully connected hidden layers on the MNIST dataset, and 
2) train a VGG Net \cite{SimonyanZ14a} on the CIFAR-10 image recognition dataset.
For both benchmark problems, we use a batchsize $128$ and weight decay $10^{-4}$.

\paragraph{Benchmark algorithms} We benchmark the following stochastic first-order algorithms: 
\begin{itemize}
  \item \emph{Baseline optimizers.} {\sgd}, {\agd} ({\sgd} with Nesterov momentum), and {\adam} \cite{kingma2014adam} with stepsize $10^{-3}$ for both MLP and VGG tasks.
  \item \emph{Hypergradient descent heuristics} \cite{baydin2018hypergradient}. {\sgdhd}, {\sgdnhd}, and {\adamhd}. We use the parameter setting in \cite{baydin2018hypergradient}: the initial stepsize is $10^{-3}$ for both MLP and VGG tasks. {\sgdhd} and {\sgdnhd} use hypergradient learning rate $10^{-3}$, and {\adamhd} uses hypergradient learning rate $10^{-7}$ for MLP and $10^{-8}$ for VGG.
  \item {\osgdh} with initial stepsize $10^{-3}$, and {\os} learning rate $10^{-3}$ for MLP; $10^{-2}$ for VGG.
\end{itemize}

\begin{figure}[htbp]
  \centering
  \includegraphics[width=\linewidth]{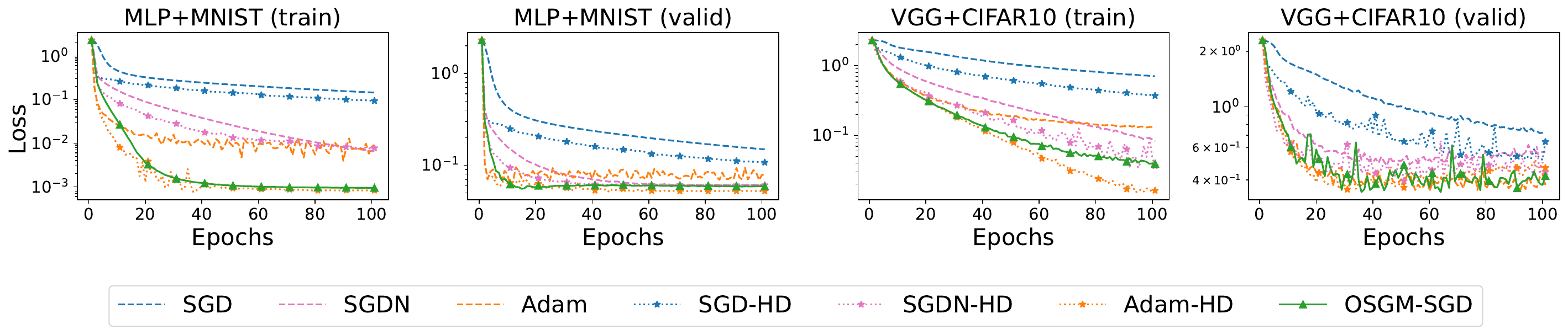}
  \caption{Performance of {\osgdh}. Left two: training and validation loss of MLP on MNIST. Right two: training and validation loss of VGG on CIFAR10.}
  \label{fig:osgm-sgd}
\end{figure}

\paragraph{Performance plots} \cref{fig:osgm-sgd} shows the training and validation loss on MLP and VGG. {\osgdh} significantly outperforms the baseline methods ({\sgd}, {\agd}, and {\adam}) on both training and validation sets. 
{\osgdh} also outperforms {\sgdhd} uniformly, showing the advantage of feedback function in {\os}. 
Notably, {\osgdh} is competitive with {\adamhd}, demonstrating that simply using adaptive stepsize for {\sgd} can match the performance of diagonal scaled momentum methods.

\section{Conclusion}

In this work, we introduce {\sosgm}, an extension of {\os} that learns a matrix stepsize for stochastic gradient methods.
We propose the {\osgd} and {\ov} algorithms, prove linear convergence for {\ov} and high-probability convergence guarantees for {\osgd} in the large-batch regime.
Numerical experiments show the advantage of these methods especially on ill-conditioned machine learning problems.

Two questions remain open: does {\ov} admit convergence guarantees when the stepsize is updated at each inner iteration rather than only in the outer loop? Can the analysis of {\osgd} be extended to show convergence even with non-vanishing gradient noise?

\renewcommand \thepart{}
\renewcommand \partname{}

\bibliographystyle{plainnat}
\bibliography{ref}

\newpage

\doparttoc
\faketableofcontents
\appendix

\addcontentsline{toc}{section}{Appendix}
\part{Appendix}
\parttoc

\section{Details of counterexample} \label{app:counterexample}

\begin{example} \label{counterexample}
  Consider an instantiation of problem \eqref{eq:F} with $n=2$:
  \begin{equation}\label{eq:counterexample-f}
    f_1 (x) = \left\{\begin{array}{ll}
       0, & x \leq 0\\
       \tfrac{1}{2} x^2 & 0 < x \leq 1\\
       x - 0.5, & x > 1
     \end{array}\right., \quad f_2 (x) = \left\{\begin{array}{ll}
       - x - 0.5, & x \leq - 1\\
       \tfrac{1}{2} x^2, & - 1 < x \leq 0\\
       0, & x > 0
     \end{array}\right..
  \end{equation}
  The minimizer of $f$ is $x^{\star} = 0$ with $f^\star=0$, and $x^{\star} \in \argmin_x f_1(x) \cap \argmin_x f_2(x)$. This problem satisfies the interpolation condition, and each $f_i$ is convex and 1-smooth.
\end{example}

\begin{figure}[ht]
  \centering
  \includegraphics[width=0.4\linewidth]{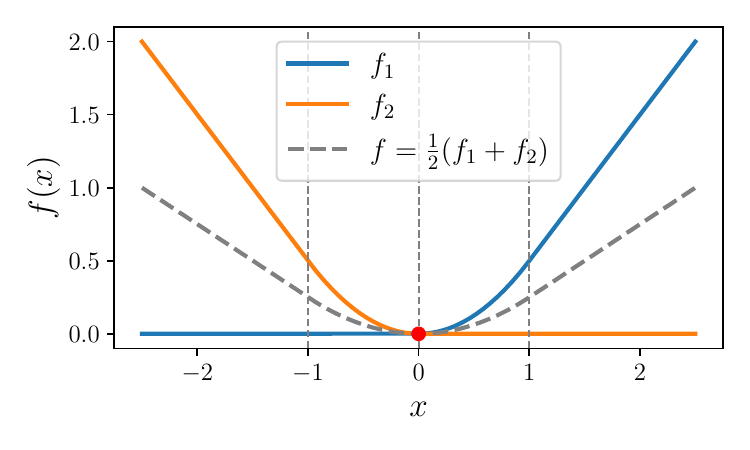}
  \includegraphics[width=0.4\linewidth]{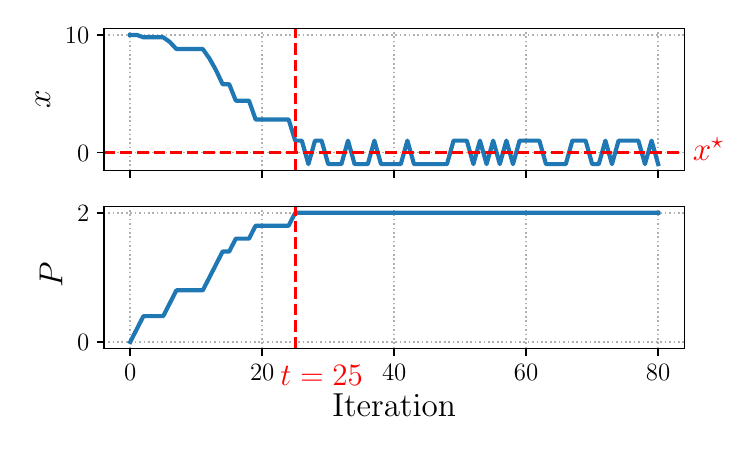}
  \caption{Illustration of \cref{counterexample}. Left: plot of $f$ defined by \eqref{eq:counterexample-f}. Right: behaviour of {\sosgm} with naive hypergradient feedback \eqref{eq:hoverfit}, with $x_0=10$, $P_0=0$, and $\eta=0.1$.}
  \label{fig:counterexample}
\end{figure}

We show that if {\os} is initialized at $x^0 = 1$, $P^0 = 2$, the {\os} iteration using the naive feedback \cref{eq:hoverfit} does not converge: instead, $\| x^k - x^{\star} \| = 1$ for every iterate $x^k$.  \\
\textit{Proof:} Start from $(x^0, P^0)$. \emph{Case 1}: $\xi^0=1$, then $\nabla f_{1} (x^0 - P_0 \nabla f_{1} (x^0))=0$. From the {\os} update with feedback \eqref{eq:hoverfit}, $P$ is not updated and we have $(x^1, P^1) = (- 1, 2)$. \emph{Case 2}: $\xi^0=2$, then $\nabla f_{2} (x^0)=0$ and $(x^1, P^1) = (1, 2)$. In either case, the stepsize $P$ is not updated and the iterate $x$ does not contract to the optimal solution $x^{\star}$.

We can generalize the proof of \cref{counterexample} to show that even if $P^0 = 0$ and the {\os} learning rate is arbitrary $\eta > 0$, there exists $x^0>1$ such that {\os} with hypergradient feedback $h_{x^k, \xi^k} (P_k)$ \eqref{eq:hoverfit} fails to converge. Since $\nabla f_2(x) =0 $ for all $x \geq 0$, whenever $\xi^k=2$ for $x^k \geq 0$, both stepsize $P_k$ and iterate $x^k$ are unchanged. So for simplicity, we will number only the iterates where $\xi^k=1$. From the {\os} update, we obtain $P_k=\eta k$, and $x^k = x^0 - \tfrac{(k-1)k}{2} \eta$. If $t=\lceil \tfrac{2}{\eta} \rceil$ and $x^0=\tfrac{(t-1)t}{2} \eta + 1$, then $x^t=1, P_t \geq 2$. Then repeating the argument in the previous paragraph concludes that the algorithm does not converge. This behaviour is also visualized in \cref{fig:counterexample}. After iteration $t=25$, stepsize $P$ stabilizes at $2$, and iterate $x$ oscillates between 1 and $-1$.

It is worth noting that the above failure mode of \eqref{eq:hoverfit} also applies when {\os} is used to tune the stepsize of VR methods. In particular, consider the ideal (but typically unavailable) VR estimator
$g^k \assign \nabla f_{\xi^k}(x^k) - \nabla f_{\xi^k}(x^\star)$ \cite{gower2020variance}.
In \cref{counterexample}, we have $\nabla f_{\xi^k}(x^\star)=0$ for every sample $\xi^k$, and hence $g^k=\nabla f_{\xi^k}(x^k)$; the resulting dynamics coincide with the stochastic gradient case above.
Therefore, applying {\os} with the naive in-sample feedback \eqref{eq:hoverfit} to update the stepsize within the inner (stochastic) iterations of a VR method can also fail to converge.
This does not conflict with {\ov}, which updates the stepsize only in the outer loop using deterministic feedback based on the full gradient at the snapshot point.

The issue in the counterexample is that the learned stepsize and resulting update can be (locally) optimal for each sampled objective $f_{\xi^k}$, while still being suboptimal for $f$. 
This suboptimality with respect to $f$ is not captured by the naive feedbacks \eqref{eq:hoverfit}, which evaluate progress only on the selected mini-batch.
We therefore require feedbacks that better reflect progress on the full objective $f$.

\section{Proofs of results in \cref{sec:large-batchsize}}

\subsection{Proof of \cref{ex:gradient-norm}} \label{app:gradient-norm-ex}
\begin{proof}
  Let $H \assign \tfrac1n \sum_{i=1}^n a_i a_i^\top$, the deviation of the stochastic gradient is
\[
    \nabla f_i(x) - \nabla f(x)
    = (a_i a_i^\top - H)(x - x^\star)
    = \Delta_i (x - x^\star).
\]
Since the feature vectors $a_i$ are sub-Gaussian with parameter $\sigma_a$, meaning
$\langle a_i, u \rangle$ is sub-Gaussian with parameter at most $\sigma_a\|u\|$ for
any $u \in \mathbb{R}^d$.  Then standard results on sub-Gaussian quadratic
forms \cite{vershynin2018high} imply the deviation bound
\[
    \mathbb{P}\!\left\{ \|\Delta_i v\| \ge t \|v\| \right\}
    \le
    2 \exp\!\big( -\tfrac{c\,t^{2}}{\sigma_a^{4}} \big)
    \qquad\forall v \neq 0,
\]
for an absolute constant $c>0$.  Applying this with $v = x - x^\star$ and
$\|\nabla f(x)\| \ge \lambda_{\min}(H)\,\|x - x^\star\|$, we obtain
\[
    \mathbb{P}\!\left\{
        \|\nabla f_i(x) - \nabla f(x)\|
        \;\ge\;
        t \, \|\nabla f(x)\|
    \right\}
    \le
    2 \exp\!\big(
        -\tfrac{c\,(t \lambda_{\min}(H))^{2}}{\sigma_a^{4}}
    \big),
\]
which is precisely \cref{as:bound-g} (iii) with $\sigma_1 = \Theta(\sigma_a^{2}/\lambda_{\min}(H))$.

Moreover, the same model also satisfies the function-value oracle
\Cref{as:bound-f}. Writing $u \assign x - x^\star$, we have
\[
    f_i(x) = \tfrac12 (a_i^\top u)^2,
    \qquad
    f(x)   = \tfrac12 u^\top H u,
\]
so
\[
    f_i(x) - f(x)
    = \tfrac12\big[(a_i^\top u)^2 - u^\top H u\big].
\]
Since $a_i^\top u$ is sub-Gaussian with parameter at most
$\sigma_a\|u\|$, the Hanson--Wright inequality
\cite{vershynin2018high} implies that
\[
    \mathbb{P}\!\left\{
        \big|(a_i^\top u)^2 - u^\top H u\big|
        \ge 2 t \, u^\top H u
    \right\}
    \le
    2 \exp\!\big(
        -\tfrac{\tilde c\,(t \lambda_{\min}(H))^{2}}{\sigma_a^{4}}
    \big),
\]
for some absolute constant $\tilde c>0$.  Using
$|f_i(x) - f(x)| = \tfrac12 |(a_i^\top u)^2 - u^\top H u|$ and
$|f(x)-f(x^\star)| = \tfrac12 u^\top H u$, this yields
\[
    \mathbb{P}\!\left\{
        |f_i(x) - f(x)|
        \;\ge\;
        t \, |f(x) - f(x^\star)|
    \right\}
    \le
    2 \exp\!\big(
        -\tfrac{\tilde c\,(t \lambda_{\min}(H))^{2}}{\sigma_a^{4}}
    \big),
\]
which matches \cref{as:bound-f} with
$\sigma_0 = \Theta(\sigma_a^{2}/\lambda_{\min}(H))$.  In particular,
both the gradient oracle and the function-value oracle have relative
noise levels of the same order, $\sigma_0 \asymp \sigma_1$.
\end{proof}

\subsection{Proof sketch and lemmas}
The proof sketch is as follows. We first establish the convergence guarantee for {\osgd} under the deterministically bounded oracles (\cref{as:bound-g-dtm,as:bound-f-dtm}). Then, notice that \cref{as:bound-g,as:bound-f} implies \cref{as:bound-g-dtm,as:bound-f-dtm} uniformly over the iterates with high probability. Applying the general reduction framework introduced in \cite{attia2024note}, we obtain the high-probability convergence guarantee for {\osgd} under \cref{as:bound-g,as:bound-f} with only a small loss in logarithmic factors.

\begin{assumption}
  \label{as:bound-g-dtm}
  The stochastic gradient oracle $\nabla f_{\xi} (x)$ satisfies, $\mathbb{E}_{\xi} [\nabla f_{\xi} (x)]={\nabla}f (x)$, and for any iterate $x \in \{ x^k \}_{k = 1, 2, \ldots}$ and any $t > 0$,
    \begin{equation*}
    \| \nabla f_{\xi} (x) - \nabla f (x) \| \leq \sigma_1 \| \nabla f (x) \|.
    \end{equation*}
\end{assumption}

\begin{assumption}
  \label{as:bound-f-dtm}
  The function value oracle $f_{\xi} (x)$ is unbiased,
    $\mathbb{E}_{\xi} [f_{\xi} (x)]$=$f (x)$, and for any iterate 
  $x \in \{x^k\}_{k=1,2,\ldots}$ and any $t>0$,
  \begin{equation*}
    | f_{\xi} (x) - f (x) | \leq \sigma_0 | f (x) - f (x^{\star}) | .
  \end{equation*}
\end{assumption}

In the following subsections, we establish the convergence guarantee for {\osgd} under the deterministically bounded oracles \cref{as:bound-g-dtm,as:bound-f-dtm}.
In \cref{sec:bs-feedback}, we introduce the proxy feedback $\hat{r}_{x, \xi}, \hat{h}_{x, \xi, \zeta}$, analyze its hindsight performance and establish its reduction to the convergence guarantee.
In \cref{sec:bs-regret}, we establish the regret bounds of doing OGD on feedback $r_{x, \xi, \zeta}, h_{x, \xi, \zeta}$ with respect to the proxy feedback. 
In \cref{sec:bs-conv}, we derive convergence guarantee under deterministic oracles follows by combining these results together.

\subsubsection{Feedback design}\label{sec:bs-feedback}

Define the proxy ratio and hypergradient loss as,
\begin{equation}\label{eq:hat-r}
  \hat{r}_{x, \xi} (P) = \tfrac{f (x - P \nabla f_{\xi} (x)) -
   f^{\star}}{f (x) - f^{\star}},\quad 
   \hat{h}_{x, \xi, \zeta} (P) =
   \tfrac{f_{\zeta}  (x - P \nabla f_{\xi} (x)) - f_{\zeta}
   (x)}{\|\nabla f (x)\|^2}.
\end{equation}
Notice that these feedbacks are only used for analysis purpose, not used in {\osgd} since their gradients are not available. We begin by analyzing the properties of feedback functions.

\begin{lemma}[Properties of feedback functions]
  \label{lem:feedback-smoothness}
  Under \cref{as:bound-g-dtm}, then for any iterate $x \in \{x^k \}_{k = 1, 2, \ldots}$ and for all $\xi$, the following statements hold.
  \begin{enumerate}
    \item $\hat{r}_{x, \xi}$ is convex and $2 L^2 (1 + \sigma_1)^2$-smooth.
    
    \item Suppose $\sigma_1 < 1$, $r_{x, \xi, \zeta}$ is convex, $2 L
    \left( L D + \tfrac{1 + \sigma_1}{1 - \sigma_1} \right)$-Lipschitz, and $2
    L^2$-smooth.
    
    \item $\hat{h}_{x, \xi, \zeta}$ is convex and $L (1 + \sigma_1)^2$-smooth.
    
    \item Suppose $\sigma_1 < 1$, $h_{x, \xi, \zeta}$ is convex, $\left(
    L D + \tfrac{1 + \sigma_1}{1 - \sigma_1} \right)$-Lipschitz, and
    $L$-smooth.
  \end{enumerate}
\end{lemma}

\begin{lemma}[Hindsight feedback]
  \label{lem:hindsight}
  Under \cref{as:bound-g-dtm},
  there exists hindsight stepsize $P_{\star}^r, P^h_{\star}$ such that
  \begin{enumerate}
    \item Ratio feedback, strongly convex. Suppose $\mu > 0$, for any $\delta
    \in (0, 1)$, w.p. $1 - \frac{\delta}{2}$,
    \begin{equation}
      \textstyle \sum_{k = 1}^K \hat{r}_{x^k, \xi^k} (P^r_{\star}) \leq \big( 1 -
      \tfrac{1}{\kappa_{\star} \left( 1 + \frac{\kappa_{\star}}{\kappa}
      \sigma_1^2 \right)} \big) K + \tfrac{2 \kappa_{\star}
      \sigma_1^3}{(\kappa + \kappa_{\star} \sigma_1^2)^2} \sqrt{2 K \log
      \tfrac{2}{\delta}} . \label{eq:ratio-hindsight}
    \end{equation}
    where $\kappa_{\star} \leq \kappa$ is the condition number by applying preconditioner $P$ such that $\tfrac{1}{\kappa_{\star}} P^{- 1}
    \preceq \nabla^2 f (x) \preceq P^{- 1}$.
    
    \item Hypergradient feedback, convex. Suppose $\mu \geq 0$, for any
    $\delta \in (0, 1)$, w.p. $1 - \delta$,
    \begin{equation}
      \textstyle \sum_{k = 1}^K \hat{h}_{x^k, \xi^k, \zeta^k} (P^h_{\star}) \leq -
      \tfrac{K}{2 L (1 + \sigma_1^2)} + \left[ \tfrac{\sigma_1^3}{L (1 +
      \sigma_1^2)^2} + \tfrac{\sigma_1 (1 + \sigma_1)}{L (1 + \sigma_1^2)}
      \right] \sqrt{2 K \log \tfrac{1}{\delta}} . \label{eq:hyper-hindsight}
    \end{equation}
  \end{enumerate}
\end{lemma}

\begin{remark}
  The hindsight convergence of ratio feedback \eqref{eq:ratio-hindsight}
  suggests that, when the underlying condition number $\kappa$ is fixed,
  improving the preconditioner, or decreasing $\kappa_{\star}$, reduces the
  impact of stochastic gradient noise $\sigma_1$.
\end{remark}

\subsubsection{Regret analysis}\label{sec:bs-regret}

Suppose $\{ P_1, \ldots, P_k \}$ is obtained from {\osgd}, we have the
following regret guarantee for both ratio and hypergradient feedback.

\begin{lemma}[Regret bounds]
  \label{lem:regret-exact-func} 
  Under \cref{as:bound-g-dtm,as:bound-f-dtm}. Then for any preconditioner $P_{\star} \in \mathcal{P}$,
  \begin{enumerate}
    \item \emph{Ratio feedback.} Suppose $\sigma_1 < 1$, and $f_{\star} = \min_x
    f_{\xi} (x)$ for all $\xi$, for any $\delta \in (0, 1)$, w.p. $1 -
    \frac{\delta}{2}$,
    \begin{align}
      \textstyle \sum_{k = 1}^K (\hat{r}_{x^k, \xi^k} (P_k) - \hat{r}_{x^k, \xi^k}
      (P_{\star})) &{}\leq 2 L D \left( L D + \tfrac{1 + \sigma_1}{1 -
      \sigma_1} \right) \sqrt{K} + 2 \sigma_0 L D \left( L D + \tfrac{1}{1 -
      \sigma_1} \right) K \nonumber\\
      &{}\quad + 2 \sigma_1 L D \left( L D + \tfrac{1}{1 - \sigma_1} \right)
      \sqrt{2 K \log \tfrac{2}{\delta}} . \label{eq:regret-ratio} 
    \end{align}
    When $\sigma_0 = 0$, the function value is exact, the regret bound is
    sublinear,
    \[ \textstyle \sum_{k = 1}^K (\hat{r}_{x^k, \xi^k} (P_k) - \hat{r}_{x^k, \xi^k}
      (P_{\star})) \leq
       2 L D \left( L D + \tfrac{1 + \sigma_1}{1 - \sigma_1} \right) \sqrt{K}
       + 2 \sigma_1 L D \left( L D + \tfrac{1}{1 - \sigma_1} \right) \sqrt{2 K
       \log \tfrac{2}{\delta}} . \]
    \item \emph{Hypergradient feedback.} Suppose $\sigma_1 < 1$,
    \begin{align}
      \textstyle \sum_{k = 1}^K (\hat{h}_{x^k, \xi^k} (P_k) - \hat{h}_{x^k, \xi^k}
      (P_{\star})) &{}\leq D \left( L D + \tfrac{1 + \sigma_1}{1 - \sigma_1}
      \right) \sqrt{K} + 3 \sigma_1 D \left( L D + \tfrac{1 + \sigma_1}{1 -
      \sigma_1} \right) K. \label{eq:regret-hyper} 
    \end{align}
  \end{enumerate}
\end{lemma}

\subsubsection{Iteration complexity}\label{sec:bs-conv}

Before establishing convergence under high-probability assumption, we first
prove convergence under the deterministic assumption.

\begin{proposition}[Convergence with deterministically bounded oracles]
  \label{prop:conv-dtm}
  Under \cref{as:bound-g-dtm,as:bound-f-dtm}. Running {\sosgm} for $K$
  iterations, then for any $\delta \in (0, 1)$, w.p. $1 - \delta$,
  \begin{enumerate}
    \item {\emph{Ratio feedback, strongly convex.}} Suppose $\mu > 0$, then
    for any $\sigma_1 < \tfrac{1}{2}$, $\sigma_0 =\mathcal{O} \left(
    \tfrac{1}{L^2 D^2 \kappa_{\star}} \right)$,
    \[ \tfrac{f (x^{K}) - f (x^{\star})}{f (x^0) - f (x^{\star})} \leq
       \left( 1 - \tfrac{1}{2 \left( 1 + \tfrac{\kappa_{\star}}{\kappa}
       \sigma_1^2  \right) \kappa_{\star}} + \mathcal{O} \left(
       \tfrac{1}{\sqrt{K}} \sqrt{\log \tfrac{1}{\delta}} \right) \right)^K .
    \]
    \item {\emph{Hypergradient feedback, strongly convex.}} Suppose $\mu > 0$,
    then for $\sigma_1 =\mathcal{O} \left( \tfrac{1}{L^2 D^2} \right)$,
    $\sigma_0 =\mathcal{O} \left( \tfrac{1}{\kappa} \right)$,
    \[ \tfrac{f (x^{K}) - f (x^{\star})}{f (x^0) - f (x^{\star})} \leq
       \left( 1 - \tfrac{1}{2 (1 + \sigma_1^2 ) \kappa} + \mathcal{O} \left(
       \tfrac{1}{\sqrt{K}} \sqrt{\log \tfrac{1}{\delta}} \right) \right)^K .
    \]
    \item {\emph{Hypergradient feedback, convex.}} Suppose $\mu = 0$, then for
    $\sigma_1 =\mathcal{O} \left( \tfrac{1}{L^2 D^2} \right)$, $\sigma_0 = 0$,
    \[ f (x^{K}) - f (x^{\star}) \leq \min \left\{ \tfrac{\Delta^2}{\max
       \left\{ K \left( \frac{1}{4 L (1 + \sigma_1^2)} - \mathcal{O} \left(
       \tfrac{1}{\sqrt{K}} \sqrt{\log \tfrac{1}{\delta}} \right) \right), 0
       \right\}}, f (x^0) - f^{\star} \right\} . \]
  \end{enumerate}
  The exact expression is shown in the proof.
\end{proposition}

\subsection{Proof of \cref{thm:conv-subG}}
Now we prove \cref{thm:conv-subG} from \cref{prop:conv-dtm}. We use the general
reduction framework from bounded stochastic oracles to light-tailed oracles.
We rephrase the result for completeness.

\begin{lemma}[Reduction framework from bounded to light-tailed oracles \cite{attia2024note}]
  \label{lem:bdd-ltail}
  Given an algorithm $\mathcal{A}$, number of rounds $K$,
  and a sub-Gaussian sampling oracle $\mathcal{O}$, there exists a $B$-bounded
  sampling oracle $\widetilde{\mathcal{O}}$ with
  \[ B = 4 \sigma \sqrt{\max \left\{ \log \tfrac{4 K}{\delta}, 1 \right\}}, \]
  such that $\mathbb{E} [\mathcal{O} (x)] =\mathbb{E} [\widetilde{\mathcal{O}}
  (x)]$ for all queries $x \in \mathcal{X}$, and with probability at least $1
  - \delta$, the outputs of algorithm $\mathcal{A}$ with $\mathcal{O}$ and
  $\widetilde{\mathcal{O}}$ are identical.
\end{lemma}

\cref{lem:bdd-ltail} implies that for analyzing an algorithm with
light-tailed oracles, it suffices to analyze a simpler version of the
algorithm that uses the bounded oracles. The results derived from bounded
oracles equally apply to the original algorithm with only a small loss in
logarithmic factors in $K$.

As a direct application of \cref{lem:bdd-ltail}, we get the convergence of {\sosgm} under sub-Gaussian Assumptions.

\begin{proof}[Proof of \cref{thm:conv-subG}]
  By \cref{lem:bdd-ltail}, substituting $\sigma \gets 4 \sigma \sqrt{\max \left\{
  \log \tfrac{4 K}{\delta}, 1 \right\}}$ for both $\sigma=\sigma_0$ and $\sigma=\sigma_1$ into
  \cref{prop:conv-dtm} finishes the proof.
\end{proof}

\subsection{Proof of lemmas}

\subsubsection{Proof of \cref{lem:feedback-smoothness}}
\begin{proof}
  Notice that $f (x - P \nabla f_{\xi } (x))$ is convex in $P$ since it is the
  composition between affine function $x - P \nabla f_{\xi} (x)$ and convex
  function $f$. Therefore $\hat{r}_{x, \xi}$ is convex because it simply
  translates and scales $f (x - P \nabla f_{\xi } (x))$ by a positive factor
  $f (x) - f^{\star}$. Similarly, $\hat{h}_{x, \xi}$ is convex.
  
  Also, $f_{\zeta}  (x - P \nabla f_{\xi} (x))$ is convex in $P$ since it is
  the composition between affine function $x - P \nabla f_{\xi} (x)$ and
  convex function $f_{\zeta}$. Therefore $r_{x, \xi, \zeta}$ and $h_{x, \xi,
  \zeta}$ are convex.
  
  Then we consider the smoothness of the feedback functions. First we prove
  that $u_{x, \xi} (P) = f (x - P \nabla f_{\xi } (x))$ is $L \| \nabla
  f_{\xi} (x)\|^2 $-smooth. For any $P_1, P_2$,
  \begin{align*}
    \| \nabla u_{x, \xi} (P_1) - \nabla u_{x, \xi} (P_2) \|_F &{}= \| (\nabla
    f (x - P_1 \nabla f_{\xi } (x)) - \nabla f (x - P_2 \nabla f_{\xi } (x)))
    \nabla f_{\xi } (x) \|_F\\
    &{}= \| \nabla f (x - P_1 \nabla f_{\xi } (x)) - \nabla f (x - P_2
    \nabla f_{\xi } (x)) \|  \| \nabla f_{\xi } (x) \|\\
    &{}\leq L \| \nabla f_{\xi } (x) \|^2  \| P_1 - P_2 \|_F,
  \end{align*}
  Similarly, let $\hat{u}_{x, \xi, \zeta} (P) = f_{\zeta} (x - P \nabla f_{\xi
  } (x))$, $\hat{u}_{x, \xi, \zeta}$ is $L \| \nabla f_{\xi} (x)\|^2 $-smooth.
  
  where the inequality is because $f$ is $L$-smooth. Notice that
  \[ \hat{r}_{x, \xi} (P) = \tfrac{u_{x, \xi} (P) - f^{\star}}{f (x) -
     f^{\star}}, \hat{h}_{x, \xi, \zeta} (P) = \tfrac{\hat{u}_{x, \xi, \zeta} (P)
     - f (x)}{\| \nabla f (x) \|^2}, h_{x, \xi, \zeta} (P) =
     \tfrac{\hat{u}_{x, \xi, \zeta} (P) - f_{\zeta} (x)}{\| \nabla f_{\xi} (x)
     \|^2}, r_{x, \xi, \zeta} (P) = \tfrac{\hat{u}_{x, \xi, \zeta} (P) -
     f^{\star}}{f_{\xi} (x) - f^{\star}} \]
  Apply the inequality $\| \nabla f_{\xi} (x)\| \leq (1 + \sigma_1) \|
  \nabla f (x) \|$ from \cref{as:bound-g-dtm}. The smoothness of $\hat{r}_{x, \xi}$ is $\tfrac{L \|
  \nabla f_{\xi} (x)\|^2}{f (x) - f^{\star}} \leq (1 + \sigma_1)^2 \tfrac{L \|
  \nabla f (x)\|^2}{f (x) - f^{\star}} \leq 2 L^2 (1 + \sigma_1)^2$. The
  smoothness of $\hat{h}_{x, \xi}$ is $\tfrac{L \| \nabla f_{\xi} (x)\|^2}{\|
  \nabla f (x) \|^2} \leq L (1 + \sigma_1)^2$. The smoothness of $r_{x, \xi,
  \zeta}$ is $2 L^2$. The smoothness of $h_{x, \xi, \zeta}$ is $L$.
  
  For the Lipschitzness of $r_{x, \xi, \zeta}$ and $h_{x, \xi, \zeta}$,
  \[ \| \nabla h_{x, \xi, \zeta} (P) \|_F \leq \tfrac{\| \nabla f_{\zeta} (x -
     P \nabla f_{\xi} (x) \|}{\| \nabla f_{\xi} (x) \|} \leq
     \tfrac{\| \nabla f_{\zeta} (x - P \nabla f_{\xi} (x)) - \nabla f_{\zeta}
     (x) \|}{\| \nabla f_{\xi} (x) \|} + \tfrac{\| \nabla f_{\zeta} (x) \|}{\|
     \nabla f_{\xi} (x) \|} \leq L D + \tfrac{1 + \sigma_1}{1 - \sigma_1} \]
  \[ \| \nabla r_{x, \xi, \zeta} (P) \|_F \leq \tfrac{\| \nabla f_{\zeta} (x -
     P \nabla f_{\xi} (x) \|}{f_{\xi} (x) - f^{\star}} \leq
     \tfrac{2 L \| \nabla f_{\zeta} (x - P \nabla f_{\xi} (x)
     \|}{\| \nabla f_{\xi} (x) \|} \leq 2 L \left( L D + \tfrac{1 +
     \sigma_1}{1 - \sigma_1} \right) . \]
  Therefore $r_{x, \xi, \zeta}$ is $2 L \left( L D + \tfrac{1 + \sigma_1}{1 -
  \sigma_1} \right)$-Lipschitz and $h_{x, \xi, \zeta}$ is $\left( L D +
  \tfrac{1 + \sigma_1}{1 - \sigma_1} \right)$-Lipschitz.
\end{proof}

\subsubsection{Proof of \cref{lem:hindsight}}

Before proving \cref{lem:hindsight}, we introduce the Azuma's
concentration inequality for martingale difference sequence (MDS).

\begin{lemma}
  \label{lem:azuma}
  Let $Z_1, \ldots, Z_K$ be an MDS, if $| Z_k | \leq B$ for
  all $k$ almost surely, then for every $\delta \in (0, 1)$,
  \[ \mathbb{P} \left( \textstyle \sum_{k = 1}^K Z_k \geq B \sqrt{2 K \log \frac{1}{\delta}}
     \right) \leq \delta . \]
\end{lemma}

\begin{proof}[Proof of \cref{lem:hindsight}]
  \textbf{Part (i) Ratio feedback, strongly convex.} Suppose there exists stepsize $P$ and condition number
  $\kappa_{\star} \leq \kappa$ such that,
  \[ \tfrac{1}{\kappa_{\star}} P^{- 1} \preceq \nabla^2 f (x) \preceq P^{- 1},
     \qquad \forall x. \]
  Let $P^r_{\star} = c P$, where $c \leq 1$ is a constant to be chosen. Then
  the following properties hold,
  \begin{equation}
    \tfrac{c}{\kappa_{\star}} I \preceq P_{\star}^{1 / 2} \nabla^2 f (x)
    P_{\star}^{1 / 2} \preceq c I \quad \text{and} \quad
    \tfrac{c}{\kappa_{\star} \mu} I \preceq P_{\star} \preceq \tfrac{c}{L} I.
    \label{eq:P-star-bound}
  \end{equation}
  Consider the descent property,
  \begin{align}
    \begin{split}\label{eq:hindsight-1} 
      &{} f (x^k - P^r_{\star} \nabla f_{\xi^k} (x^k)) - f (x^k) \\
    =&{} - \langle \nabla f (x^k), P^r_{\star} \nabla f_{\xi^k} (x^k)
    \rangle + \tfrac{1}{2} \langle P^r_{\star} \nabla f_{\xi^k} (x^k),
    \nabla^2 f (\psi^k) P^r_{\star} \nabla f_{\xi^k} (x^k) \rangle \\
    \leq &{} - \| \nabla f (x^k) \|_{P^r_{\star}}^2 - \langle \nabla
    f_{\xi^k} (x^k) - \nabla f (x^k), P^r_{\star} \nabla f (x^k) \rangle +
    \tfrac{c}{2} \| \nabla f_{\xi^k} (x^k) \|_{P^r_{\star}}^2 \\
    = &{} - \left( 1 - \tfrac{c}{2} \right) \| \nabla f (x^k)
    \|_{P^r_{\star}}^2 + \tfrac{c}{2} \| \nabla f_{\xi^k} (x^k) - \nabla f
    (x^k) \|_{P^r_{\star}}^2 + (c - 1) \langle \nabla f_{\xi^k} (x^k) - \nabla
    f (x^k), P^r_{\star} \nabla f (x^k) \rangle \\
    \leq &{} - \left( 1 - \tfrac{c}{2} \right) \| \nabla f (x^k)
    \|_{P^r_{\star}}^2 + \tfrac{c^2 \sigma_1^2}{2 L} \| \nabla f_{\xi^k} (x^k)
    - \nabla f (x^k) \|_2^2 + (c - 1) \langle \nabla f_{\xi^k} (x^k) - \nabla
    f (x^k), P^r_{\star} \nabla f (x^k) \rangle \\
    \leq &{} - \left( 1 - \tfrac{c}{2} - \tfrac{c \kappa_{\star}
    \sigma_1^2}{2 \kappa} \right) \| \nabla f (x^k) \|_{P^r_{\star}}^2 +
    \tfrac{(c - 1) c}{L} \langle \nabla f_{\xi^k} (x^k) - \nabla f (x^k),
    \nabla f (x^k) \rangle 
    \end{split}
  \end{align}
  where the first equality is by Taylor expansion, the second inequality is by
  \eqref{eq:P-star-bound}, the third equality expands $\| \nabla f_{\xi^k}
  (x^k) \|_{P^r_{\star}}^2 = \| \nabla f (x^k) + (\nabla f_{\xi^k} (x^k) -
  \nabla f (x^k)) \|_{P^r_{\star}}^2$, the fourth inequality uses $P^r_{\star}
  \preceq \tfrac{c}{L} I$ in \eqref{eq:P-star-bound}, and the
  last inequality is by \eqref{eq:P-star-bound}.
  
  Since $\frac{c}{\kappa_{\star}} {I \preceq P^r_{\star}}^{1 / 2} \nabla^2 f
  (x) {P^r_{\star}}^{1 / 2}$,
  \[ f (x) - f^{\star} \leq \frac{\kappa_{\star}}{2 c} \| \nabla f (x)
     \|_{P_{\star}}^2 . \]
  Then, substitute this inequality back into \eqref{eq:hindsight-1},
  \begin{align*}
    f (x^k - P^r_{\star} \nabla f_{\xi^k} (x^k)) - f (x^k) &{}\leq - \left(
    \tfrac{2 c}{\kappa_{\star}} - \tfrac{c^2}{\kappa_{\star}} - \tfrac{c^2
    \sigma_1^2}{\kappa} \right) (f (x^k) - f^{\star})\\
    &{}\quad + \tfrac{(c - 1) c}{L} \langle \nabla f_{\xi^k} (x^k) - \nabla f
    (x^k), \nabla f (x^k) \rangle
  \end{align*}
  Choose $c = \frac{1}{1 + (\kappa_{\star} / \kappa) \sigma_1^2}$, and
  rearrange,
  \begin{equation}
    \hat{r}_{x^k, \xi^k} (P^r_{\star}) \leq 1 - \tfrac{1}{\kappa_{\star}
    \left( 1 + \frac{\kappa_{\star}}{\kappa} \sigma_1^2 \right)} -
    \tfrac{\frac{\kappa_{\star}}{\kappa} \sigma_1^2}{\left[ 1 +
    \frac{\kappa_{\star}}{\kappa} \sigma_1^2 \right]^2 L} \tfrac{\langle
    \nabla f_{\xi^k} (x^k) - \nabla f (x^k), \nabla f (x^k) \rangle}{f (x^k) -
    f (x^{\star})} . \label{eq:hindsight-ratio-telescope}
  \end{equation}
  Define
  \[ Z_k = \tfrac{\langle \nabla f_{\xi^k} (x^k) - \nabla f (x^k), \nabla f
     (x^k) \rangle}{f (x^k) - f (x^{\star})}, \]
  This is an MDS, and $| Z_k | \leq \tfrac{\sigma_1 \| \nabla f (x^k) \|^2}{f
  (x^k) - f (x^{\star})} \leq 2 L \sigma_1$ for all $k$. By
  \cref{lem:azuma}, w.p. at least $1 - \frac{\delta}{2}$,
  \begin{equation}
    \textstyle \sum_{k = 1}^K Z_k \leq 2 L \sigma_1 \sqrt{2 K \log \tfrac{2}{\delta}} .
    \label{eq:hindsight-ratio-concentration}
  \end{equation}
  Telescope \eqref{eq:hindsight-ratio-telescope} and use concentration
  inequality \eqref{eq:hindsight-ratio-concentration} and gets
  \eqref{eq:ratio-hindsight}.
  
  \textbf{Part (ii) Hypergradient feedback, convex.} Consider the descent property with scalar stepsize
  $P_{\star}^h = \alpha I$ where $\alpha$ is to be chosen,
  \begin{align*}
    &{} f_{\zeta^k} (x^k - \alpha \nabla f_{\xi^k} (x^k)) - f_{\zeta^k}
    (x^k) \\
    \leq &{} - \alpha \langle \nabla f_{\xi^k} (x^k), \nabla f_{\zeta^k}
    (x^k) \rangle + \tfrac{\alpha^2 L}{2} \| \nabla f_{\xi^k} (x^k) \|^2\\
    = &{} - \alpha \langle \nabla f_{\xi^k} (x^k), \nabla f (x^k) \rangle -
    \alpha \langle \nabla f_{\xi^k} (x^k), \nabla f_{\zeta^k} (x^k) - \nabla f
    (x^k) \rangle + \tfrac{\alpha^2 L}{2} \| \nabla f_{\xi^k} (x^k) \|^2\\
    = &{} (- \alpha + \alpha^2 L) \langle \nabla f_{\xi^k} (x^k), \nabla f
    (x^k) \rangle + \tfrac{\alpha^2 L}{2} \| \nabla f_{\xi^k} (x^k) - \nabla f
    (x^k) \|^2 - \tfrac{\alpha^2 L}{2} \| \nabla f (x^k) \|^2\\
    &{}\quad - \alpha \langle \nabla f_{\xi^k} (x^k), \nabla f_{\zeta^k} (x^k) -
    \nabla f (x^k) \rangle\\
    =&{} \left( \tfrac{\alpha^2 L \sigma_1^2}{2} + \tfrac{\alpha^2 L}{2} -
    \alpha \right) \| \nabla f (x^k) \|^2 + (- \alpha + \alpha^2 L) \langle
    \nabla f_{\xi^k} (x^k) - \nabla f (x^k), \nabla f (x^k) \rangle\\
    &{}\quad - \alpha \langle \nabla f_{\xi^k} (x^k), \nabla f_{\zeta^k} (x^k) -\nabla f (x^k) \rangle
  \end{align*}
  where the first equality is due to,
  \begin{align*}
    \| \nabla f_{\xi^k} (x^k) \|^2 &{}= \| \nabla f_{\xi^k} (x^k) - \nabla f
    (x^k) \|^2 - \| \nabla f (x^k) \|^2 + 2 \langle \nabla f_{\xi^k} (x^k),
    \nabla f (x^k) \rangle .
  \end{align*}
  Let $\alpha = \frac{1}{L (1 + \sigma_1^2) }$, and rearrange,
  \begin{equation}
    \hat{h}_{x^k, \xi^k} (P^h_{\star}) \leq - \tfrac{1}{2 L (1 + \sigma_1^2)}
    - \tfrac{\sigma_1^2}{L (1 + \sigma_1^2)^2} \tfrac{\langle \nabla f_{\xi^k}
    (x^k) - \nabla f (x^k), \nabla f (x^k) \rangle}{\| \nabla f (x^k) \|^2} -
    \tfrac{\langle \nabla f_{\xi^k} (x^k), \nabla f_{\zeta^k} (x^k) - \nabla f
    (x^k) \rangle}{L (1 + \sigma_1^2) \| \nabla f (x^k) \|^2}
    \label{eq:hindsight-hyper-telescope}
  \end{equation}
  Let $Z^{(1)}_k = \frac{\langle \nabla f (x^k, \xi^k) - \nabla f (x^k),
  \nabla f (x^k) \rangle}{\| \nabla f (x^k) \|^2}$. This is an MDS, and $|
  Z^{(1)}_k | \leq \sigma_1$ for all $k$. By \cref{lem:azuma}, w.p. at
  least $1 - \frac{\delta}{2}$,
  \begin{equation}
    \textstyle \sum_{k = 1}^K Z^{(1)}_k \leq \sigma_1 \sqrt{2 K \log \tfrac{2}{\delta}} .
    \label{eq:hindsight-hyper-concentration}
  \end{equation}
  Let $Z^{(2)}_k = \frac{\langle \nabla f_{\xi^k} (x^k), \nabla f_{\zeta^k}
  (x^k) - \nabla f (x^k) \rangle}{\| \nabla f (x^k) \|^2}$. This is an MDS,
  and $| Z^{(2)}_k | \leq \sigma_1 (1 + \sigma_1)$ for all $k$. By
  \cref{lem:azuma}, w.p. at least $1 - \frac{\delta}{2}$,
\begin{equation}
    \textstyle \sum_{k = 1}^K Z^{(2)}_k \leq \sigma_1 (1 + \sigma_1) \sqrt{2 K \log
    \tfrac{2}{\delta}} . \label{eq:hindsight-hyper-concentration2}
  \end{equation}
  Telescope \eqref{eq:hindsight-hyper-telescope} and use concentration
  inequality \eqref{eq:hindsight-hyper-concentration},
  \eqref{eq:hindsight-hyper-concentration2} gets \eqref{eq:hyper-hindsight}.
\end{proof}

\subsubsection{Proof of \cref{lem:regret-exact-func}}
\begin{proof}
  \textbf{Part (i) Regret of ratio feedback.} By the convexity of $\hat{r}_{x^k, \xi^k}$, we have
  \begin{align}
    \hat{r}_{x^k, \xi^k} (P_k) - \hat{r}_{x^k, \xi^k} (P_{\star}) &{}\leq
    \langle \nabla \hat{r}_{x^k, \xi^k} (P_k), P_k - P_{\star} \rangle
    \nonumber\\
    &{}= \langle \nabla r_{x^k, \xi^k, \zeta^k} (P_k), P_k - P_{\star}
    \rangle + \langle \nabla \hat{r}_{x^k, \xi^k} (P_k) - \nabla r_{x^k,
    \xi^k, \zeta^k} (P_k), P_k - P_{\star} \rangle \nonumber\\
    &{}= \tfrac{1}{2 \eta} \| P_k - P_{\star} \|_F^2 - \tfrac{1}{2 \eta} \|
    P_{k + 1} - P_{\star} \|_F^2 + \tfrac{\eta}{2} \| \nabla r_{x^k, \xi^k,
    \zeta^k} (P_k) \|_F^2 \nonumber\\
    &{}\quad + \langle \nabla \hat{r}_{x^k, \xi^k} (P_k) - \nabla r_{x^k, \xi^k,
    \zeta^k} (P_k), P_k - P_{\star} \rangle \label{eq:regret-ratio-tel} 
  \end{align}
  Consider the last term,
  \begin{align*}
    \nabla \hat{r}_{x^k, \xi^k} (P_k) - \nabla r_{x^k, \xi^k, \zeta^k} (P_k) &{}= \tfrac{[\nabla f (x^k - P \nabla f_{\xi^k} (x^k)) - \nabla f_{\zeta^k}
    (x^k - P \nabla f_{\xi^k} (x^k))] \nabla f_{\xi^k} (x^k)^{\top}}{f_{\xi^k}
    (x^k) - f^{\star}}\\
    &{}\quad + \tfrac{\nabla f (x^k - P \nabla f_{\xi^k} (x^k)) \nabla f_{\xi^k}
    (x^k)^{\top}}{f_{\xi^k} (x^k) - f^{\star}} \left( \tfrac{f_{\xi^k} (x^k) -
    f^{\star}}{f (x^k) - f^{\star}} - 1 \right)
  \end{align*}
  Let $Z^{(1)}_k = \left\langle \tfrac{[\nabla f (x^k - P \nabla f_{\xi^k}
  (x^k)) - \nabla f_{\zeta^k} (x^k - P \nabla f_{\xi^k} (x^k))] \nabla
  f_{\xi^k} (x^k)^{\top}}{f_{\xi^k} (x^k) - f^{\star}}, P_k - P_{\star}
  \right\rangle$, $\{ Z^{(1)}_k \}$ is an MDS, and has a uniform bound,
  \begin{align*}
    | Z^{(1)}_k | &{}\leq \tfrac{\sigma_1 \| \nabla f (x^k - P \nabla
    f_{\xi^k} (x^k)) \| \| \nabla f_{\xi^k} (x^k) \|}{f_{\xi^k} (x^k) -
    f^{\star}} \| P_k - P_{\star} \|_F \leq 2 \sigma_1 L D \tfrac{\| \nabla f (x^k - P \nabla f_{\xi^k}
    (x^k)) \|}{\| \nabla f_{\xi^k} (x^k) \|} \leq 2 \sigma_1 L D \left( L D +
    \tfrac{1}{1 - \sigma_1} \right) .
  \end{align*}
  By Azuma's inequality (\cref{lem:azuma}), w.p. $1 - \tfrac{\delta}{2}$,
  \begin{align*}
    \textstyle \sum_{k = 1}^K Z^{(1)}_k &{}\leq 2 \sigma_1 L D \left( L D + \tfrac{1}{1
    - \sigma_1} \right) \sqrt{2 K \log \tfrac{2}{\delta}} .
  \end{align*}
  Let $Z^{(2)}_k = \left\langle \tfrac{\nabla f (x^k - P \nabla f_{\xi^k}
  (x^k)) \nabla f_{\xi^k} (x^k)^{\top}}{f_{\xi^k} (x^k) - f^{\star}} \left(
  \tfrac{f_{\xi^k} (x^k) - f^{\star}}{f (x^k) - f^{\star}} - 1 \right), P_k -
  P_{\star} \right\rangle$,
  \begin{align*}
    | Z^{(2)}_k | &{}\leq \tfrac{\| \nabla f (x^k - P \nabla f_{\xi^k} (x^k))
    \| \| \nabla f_{\xi^k} (x^k) \|}{f_{\xi^k} (x^k) - f^{\star}} \left|
    \tfrac{f_{\xi^k} (x^k) - f^{\star}}{f (x^k) - f^{\star}} - 1 \right| \|
    P_k - P_{\star} \|_F\\
    &{}\leq 2 \sigma_0 L D \tfrac{\| \nabla f (x^k - P \nabla f_{\xi^k}
    (x^k)) \|}{\| \nabla f_{\xi^k} (x^k) \|} \leq 2 \sigma_0 L D \left( L D +
    \tfrac{1}{1 - \sigma_1} \right) .
  \end{align*}
  Therefore, $\textstyle \sum_{k = 1}^K Z^{(2)}_k \leq 2 \sigma_0 L D \left( L D +
  \tfrac{1}{1 - \sigma_1} \right) K.$ Telescope \eqref{eq:regret-ratio-tel},
  w.p. $1 - \frac{\delta}{2}$,
  \begin{align*}
    \textstyle \sum_{k = 1}^K (\hat{r}_{x^k, \xi^k} (P_k) - \hat{r}_{x^k, \xi^k}
    (P_{\star})) &{}\leq \frac{1}{2 \eta} \| P_1 - P_{\star} \|_F^2 + 2 \eta
    L^2 \left( L D + \tfrac{1 + \sigma_1}{1 - \sigma_1} \right)^2 K + 2
    \sigma_0 L D \left( L D + \tfrac{1}{1 - \sigma_1} \right) K\\
    &{}\quad + 2 \sigma_1 L D \left( L D + \tfrac{1}{1 - \sigma_1} \right)
    \sqrt{2 K \log \tfrac{2}{\delta}} .
  \end{align*}
  Let $\eta = \frac{\| P_1 - P_{\star} \|_F}{2 L \left( \frac{1 + \sigma_1}{1
  - \sigma_1} + L D \right) \sqrt{K}}$, we derived the final regret bound \eqref{eq:regret-ratio}.
  
  \textbf{Part (ii) Regret of hypergradient feedback.} By the convexity of $\hat{h}_{x^k, \xi^k}$, we have
  \begin{align}
    \hat{h}_{x^k, \xi^k, \zeta^k} (P_k) - \hat{h}_{x^k, \xi^k, \zeta^k}
    (P_{\star}) &{}\leq \langle \nabla \hat{h}_{x^k, \xi^k} (P_k), P_k - P_{\star}
    \rangle \nonumber\\
    &{}= \langle \nabla h_{x^k, \xi^k, \zeta^k} (P_k), P_k - P_{\star}
    \rangle + \langle \nabla \hat{h}_{x^k, \xi^k, \zeta^k} (P_k) - \nabla
    \hat{h}_{x^k, \xi^k, \zeta^k} (P_k), P_k - P_{\star} \rangle \nonumber\\
    &{}\leq \tfrac{1}{2 \eta} \| P_k - P_{\star} \|_F^2 - \tfrac{1}{2 \eta}
    \| P_{k + 1} - P_{\star} \|_F^2 + \tfrac{\eta}{2} \| \nabla h_{x^k, \xi^k,
    \zeta^k} (P_k) \|_F^2 \nonumber\\
    &{}\quad + \langle \hat{h}_{x^k, \xi^k, \zeta^k} (P_k) - \nabla h_{x^k,
    \xi^k, \zeta^k} (P_k), P_k - P_{\star} \rangle \label{eq:regret-hyper-tel}
  \end{align}

  For the last term,
  \begin{align*}
    \nabla \hat{h}_{x^k, \xi^k, \zeta^k} (P_k) - \nabla h_{x^k, \xi^k,
    \zeta^k} (P_k) &{}= \tfrac{\nabla f_{\zeta^k} (x^k - P \nabla f_{\xi^k}
    (x^k)) \nabla f_{\xi^k} (x^k)^{\top}}{\| \nabla f_{\xi^k} (x^k) \|^2}
    \left( 1 - \tfrac{\| \nabla f_{\xi^k} (x^k) \|^2}{\| \nabla f (x^k) \|^2}
    \right)
  \end{align*}
  Let $Z^{(2)}_k = \left( 1 - \tfrac{\| \nabla f_{\xi^k} (x^k) \|^2}{\| \nabla
  f (x^k) \|^2} \right) \left\langle \tfrac{\nabla f_{\zeta^k} (x^k - P \nabla
  f_{\xi^k} (x^k)) \nabla f_{\xi^k} (x^k)^{\top}}{\| \nabla f_{\xi^k} (x^k)
  \|^2}, P_k - P_{\star} \right\rangle$, and
  \[ | Z^{(2)}_k | \leq (\sigma_1^2 + 2 \sigma_1) D \left( L D + \tfrac{1}{1 -
     \sigma_1} \right) \leq 3 \sigma_1 D \left( L D + \tfrac{1 + \sigma_1}{1 -
     \sigma_1} \right) . \]
  Therefore, $\textstyle \sum_{k = 1}^K Z^{(2)}_k \leq 3 \sigma_1 D \left( L D + \tfrac{1
  + \sigma_1}{1 - \sigma_1} \right) K.$ Telescope \eqref{eq:regret-hyper-tel},
  \begin{align*}
    \textstyle \sum_{k = 1}^K (\hat{h}_{x^k, \xi^k} (P_k) - \hat{h}_{x^k, \xi^k}
    (P_{\star})) &{}\leq \tfrac{1}{2 \eta} \| P_1 - P_{\star} \|_F^2 +
    \tfrac{\eta}{2} \left( L D + \tfrac{1 + \sigma_1}{1 - \sigma_1} \right)^2
    K + 3 \sigma_1 D \left( L D + \tfrac{1 + \sigma_1}{1 - \sigma_1} \right)
    K.
  \end{align*}
  Let $\eta = \frac{\| P_1 - P_{\star} \|_F}{\left( L D + \frac{1 +
  \sigma_1}{1 - \sigma_1} \right) \sqrt{K}}$, we derived the final regret
  bound \eqref{eq:regret-hyper}.
\end{proof}

\subsubsection{Proof of \cref{prop:conv-dtm}}
\begin{proof}[Proof of \cref{prop:conv-dtm}]
  \textbf{Part (i) Ratio feedback, strongly convex.} By the definiton of ratio feedback,
  \begin{align*}
    \tfrac{f (x^{K}) - f (x^{\star})}{f (x^0) - f (x^{\star})} &{}=
    \textstyle \prod_{k = 0}^{K - 1} \hat{r}_{x^k, \xi^k} (P_k) \leq \left(
    \tfrac{1}{K} \textstyle \sum_{k = 0}^{K - 1} \hat{r}_{x^k, \xi^k} (P_k)
    \right)^K
  \end{align*}
  Apply \cref{lem:hindsight} (i), \cref{lem:regret-exact-func} (i) and by union bound, w.p. $1 - \delta$,
  \begin{align*}
    \tfrac{f (x^{K}) - f (x^{\star})}{f (x^0) - f (x^{\star})} &{}\leq
    \left( 1 - \tfrac{1}{\kappa_{\star} \left( 1 +
    \frac{\kappa_{\star}}{\kappa} \sigma_1^2 \right)} + 2 \sigma_0 L D \left(
    L D + \tfrac{1}{1 - \sigma_1} \right) + \tfrac{C}{\sqrt{K}} \right)^K
  \end{align*}
  where $C = 2 L D \left( L D + \tfrac{1 + \sigma_1}{1 - \sigma_1} \right) + 2
  \sigma_1 L D \left( L D + \tfrac{1}{1 - \sigma_1} \right) \sqrt{2 \log
  \tfrac{2}{\delta}}$. To ensure linear convergence, it suffices to impose $2
  \sigma_0 L D \left( L D + \tfrac{1}{1 - \sigma_1} \right) \leq \tfrac{1}{2
  (1 + \sigma_1^2) \kappa_{\star}}$. Let $\sigma_1 \leq \tfrac{1}{2}$, the
  condition reduced to $\sigma_0 \leq \tfrac{1}{5 L D (L D + 2)
  \kappa_{\star}} = \mathcal{O} \left( \tfrac{1}{L^2 D^2 \kappa_{\star}}
  \right)$. Thus completes the proof of (i).
  
  \textbf{Part (ii) Hypergradient feedback, strongly convex.} By null step and the definition of $\hat{h}_{x^k,
  \xi^k, \zeta^k} (P_k)$,
  \begin{align*}
    f_{\zeta^k} (x^{k + 1}) - f_{\zeta^k} (x^k) &{}= \min \{ \hat{h}_{x^k,
    \xi^k, \zeta^k} (P_k), 0 \} \| \nabla f (x^k) \|^2\\
    &{}\leq 2 \mu \min \{ \hat{h}_{x^k, \xi^k, \zeta^k} (P), 0 \} (f (x^k) -
    f^{\star})
  \end{align*}
  Thus we have,
  \begin{align}
    f (x^{k + 1}) - f (x^k) &{}\leq 2 \mu \min \{ \hat{h}_{x^k, \xi^k,
    \zeta^k} (P_k), 0 \} (f (x^k) - f^{\star}) + f (x^{k + 1}) - f_{\zeta^k}
    (x^{k + 1}) + f_{\zeta^k} (x^k) - f (x^k) \nonumber\\
    &{}\leq 2 \mu \min \{ \hat{h}_{x^k, \xi^k, \zeta^k} (P_k), 0 \} (f (x^k)
    - f^{\star}) + \sigma_0 [(f (x^{k + 1}) - f^{\star}) + (f (x^k) -
    f^{\star})] \label{eq:conv-hyper-sc-single} \\
    &{}\leq \left( 2 \mu \min \{ \hat{h}_{x^k, \xi^k, \zeta^k} (P_k), 0 \} +
    \tfrac{2 \sigma_0}{1 - \sigma_0} \right) (f (x^k) - f^{\star}) \nonumber
  \end{align}
  The last inequality is because,
  \begin{align*}
    f (x^{k + 1}) - f (x^{\star}) &{}\leq [1 + 2 \mu \min \{ \hat{h}_{x^k,
    \xi^k, \zeta^k} (P_k), 0 \}] (f (x^k) - f^{\star}) + \sigma_0 [f (x^{k +
    1}) - f^{\star} + f (x^k) - f^{\star}]\\
    &{}\leq (f (x^k) - f^{\star}) + \sigma_0 [f (x^{k + 1}) - f^{\star} + f
    (x^k) - f^{\star}],
  \end{align*}
  and when $\sigma_0 < 1$, $\frac{f (x^{k + 1}) - f (x^{\star})}{f (x^k) - f
  (x^{\star})} \leq \frac{1 + \sigma_0}{1 - \sigma_0}$. Combine with
  \eqref{eq:conv-hyper-sc-single} and rearrange,
  \[ \tfrac{f (x^{k + 1}) - f^{\star}}{f (x^k) - f^{\star}} \leq 1 + 2 \mu
     \min \{ \hat{h}_{x^k, \xi^k, \zeta^k} (P_k), 0 \} + \tfrac{2 \sigma_0}{1
     - \sigma_0} . \]
  Telescope and by arithemetic mean inequality,
  \begin{align*}
    \tfrac{f (x^{K}) - f^{\star}}{f (x^0) - f^{\star}} &{}\leq \textstyle
    \prod_{k = 0}^{K - 1} \left( 1 + 2 \mu \min \{ \hat{h}_{x^k, \xi^k,
    \zeta^k} (P_k), 0 \} +
    \tfrac{2 \sigma_0}{1 - \sigma_0} \right)\\
    &{}\leq \left( \tfrac{1}{K} \textstyle \sum_{k = 0}^{K - 1} \left( 1 + 2
    \mu \min \{ \hat{h}_{x^k, \xi^k, \zeta^k} (P_k), 0 \} + \tfrac{2
    \sigma_0}{1 - \sigma_0} \right) \right)^K
  \end{align*}
  By \cref{lem:hindsight} (ii) and \cref{lem:regret-exact-func} (ii), w.p. $1 - \delta$,
  \begin{align*}
    \tfrac{f (x^{K}) - f^{\star}}{f (x^0) - f^{\star}} &{}\leq \left( 1 +
    2 \mu \min \left\{ - \tfrac{1}{2 L (1 + \sigma_1^2)} + 3 \sigma_1 D \left(
    L D + \tfrac{1}{1 - \sigma_1} \right), 0 \right\} + \tfrac{2 \sigma_0}{1 -
    \sigma_0} + \tfrac{C}{\sqrt{K}} \right)^K
  \end{align*}
  where $C = \left[ \tfrac{\sigma_1^3}{L (1 + \sigma_1^2)^2} + \tfrac{\sigma_1
  (1 + \sigma_1)}{L (1 + \sigma_1^2)} \right] \sqrt{2 \log \tfrac{1}{\delta}}
  + D \left( L D + \tfrac{1 + \sigma_1}{1 - \sigma_1} \right)$. To ensure
  linear convergence, it suffices to impose that,
  \[ 3 \sigma_1 D \left( L D + \tfrac{1}{1 - \sigma_1} \right) \leq
     \tfrac{1}{8 L (1 + \sigma_1^2)}, \quad \tfrac{2 \sigma_0}{1 - \sigma_0}
     \leq \tfrac{1}{8 \kappa (1 + \sigma_1^2)} . \]
  With $\sigma_1 \leq \tfrac{1}{2}$, it suffices to let $\sigma_0 \leq
  \tfrac{1}{20 \kappa + 1}, \sigma_1 \leq \tfrac{1}{30 L D (L D + 2)}$. Thus
  completes the proof of part (ii).
  
  \textbf{Part (iii) Hypergradient feedback, convex.} Since $\sigma_0 = 0$, by the definition of
  $\hat{h}_{x^k, \xi^k, \zeta^k} (P_k)$,
  \begin{align*}
    f (x^{k + 1}) - f (x^k) &{}= \min \{ \hat{h}_{x^k, \xi^k, \zeta^k} (P_k),
    0 \} \| \nabla f (x^k) \|^2\\
    &{}\leq \min \{ \hat{h}_{x^k, \xi^k, \zeta^k} (P_k), 0 \} \tfrac{[f
    (x^k) - f^{\star}]^2}{\| x^k - x^{\star} \|^2}\\
    &{}\leq \min \{ \hat{h}_{x^k, \xi^k, \zeta^k} (P_k), 0 \} \tfrac{[f
    (x^k) - f^{\star}]^2}{\Delta^2}
  \end{align*}
  where $\Delta \assign \max_{x \in \{ x : f (x) \leq f (x^0) \}}
  \min_{x^{\star} \in \mathcal{X}^{\star}} \| x - x^{\star} \|$. Divide by $[f
  (x^{k + 1}) - f^{\star}] [f (x^k) - f^{\star}]$ on both sides and by null
  step that $f (x^{k + 1}) \leq f (x^k)$,
  \[ \tfrac{1}{f (x^k) - f^{\star}} - \tfrac{1}{f (x^{k + 1}) - f^{\star}}
     \leq \tfrac{\min \{ \hat{h}_{x^k, \xi^k, \zeta^k} (P_k), 0 \}}{\Delta^2}
     . \]
  Telescope and rearrange,
  \[ f (x^{K}) - f^{\star} \leq \min \left\{ \tfrac{\Delta^2}{\max \left\{
     - \textstyle \sum_{k = 0}^{K - 1} \hat{h}_{x^k, \xi^k, \zeta^k} (P_k), 0
     \right\}}, f (x^0) - f^{\star} \right\} \]
  By \cref{lem:hindsight} (ii) and \cref{lem:regret-exact-func} (ii), w.p. $1 -
  \delta$,
  \[ f (x^{K}) - f^{\star} \leq \min \left\{ \tfrac{\Delta^2}{\max \left\{
     K \left( \frac{1}{2 L (1 + \sigma_1^2)} - 3 \sigma_1 D \left( L D +
     \tfrac{1 + \sigma_1}{1 - \sigma_1} \right) - \tfrac{C}{\sqrt{K}} \right),
     0 \right\}}, f (x^0) - f^{\star} \right\} \]
  where $C = \left[ \tfrac{\sigma_1^3}{L (1 + \sigma_1^2)^2} + \tfrac{\sigma_1
  (1 + \sigma_1)}{L (1 + \sigma_1^2)} \right] \sqrt{2 \log \tfrac{1}{\delta}}
  + D \left( L D + \tfrac{1 + \sigma_1}{1 - \sigma_1} \right)$. To get
  convergence, it suffices to impose
  \[
    3 \sigma_1 D \left( L D + \tfrac{1 + \sigma_1}{1 - \sigma_1} \right)
    \leq \frac{1}{4 L (1 + \sigma_1^2)}, \qquad
    \sigma_1 \leq \frac{1}{15 L D (L D + 3)}
    = \mathcal{O} \left( \tfrac{1}{L^2 D^2} \right).
  \]
  thus proved part (iii).
\end{proof}

\section{Proofs of results in \cref{sec:vr}}

\subsection{Proof of \cref{lem:vr-pot}}
\begin{proof}
  By the definition of {\svrg} gradient estimator $g_t$, and $L$-smoothness,
  \begin{align}
    \mathbb{E}_t [\| g_t \|^2] &{}= \mathbb{E}_t [\| \nabla f_{\xi^t} (x^t) -
    \nabla f_{\xi^t} (\tilde{x}^k) + \nabla f (\tilde{x}^k) \|^2] \nonumber\\
    &{}\leq 2\mathbb{E}_t [\| \nabla f_{\xi^t} (x^t) - \nabla f_{\xi^t}
    (x^{\star}) \|^2] + 2\mathbb{E}_t [\| \nabla f_{\xi^t} (\tilde{x}^k) -
    \nabla f_{\xi^t} (x^{\star}) - \nabla f (\tilde{x}^k) \|^2] \nonumber\\
    &{}\leq 2\mathbb{E}_t [\| \nabla f_{\xi^t} (x^t) - \nabla f_{\xi^t}
    (x^{\star}) \|^2] + 2\mathbb{E}_t [\| \nabla f_{\xi^t} (\tilde{x}^k) -
    \nabla f_{\xi^t} (x^{\star}) \|^2] \nonumber\\
    &{}\leq 4 L (f (x^t) - f^{\star} + f (\tilde{x}^k) - f^{\star}) 
    \label{eq:svrg-var}
  \end{align}
  \textbf{Part (i) No momentum.} The update rule is $x^{t + 1} = x^t - c P_k
  g_t$. By $\underline{\alpha} I \preceq P_k \preceq \bar{\alpha} I$ and
  \eqref{eq:svrg-var},
  \begin{align*}
    \mathbb{E}_t [f (x^{t + 1})] &{}\leq f (x^t) - c \nabla f (x^t)^{\top}
    P_k \nabla f (x^t) + \tfrac{c^2 L}{2} \mathbb{E} [\| P_k g_t \|^2]\\
    &{}\leq f (x^t) - c \, \underline{\alpha}  \| \nabla f (x^t) \|^2 +
    \tfrac{c^2 \bar{\alpha}^2 L}{2} \mathbb{E} [\| g_t \|^2]\\
    &{}= f (x^t) - \left( 2 c \, \underline{\alpha} \,  \mu - 2 c^2 \bar{\alpha}^2
    L^2 \right) (f (x^t) - f^{\star}) + 2 c^2 \bar{\alpha}^2 L^2 (f
    (\tilde{x}^k) - f^{\star}) .
  \end{align*}
  Let $V^t \assign f (x^t) - f^{\star}$ proves \eqref{pot-svrg}.

  \textbf{Part (ii) Bounded momentum.} Define $d^t \assign x^t - x^{t - 1}$.
  Since $f$ is $L$-smooth,
  \begin{align}
    \mathbb{E}_t [f (x^{t + 1})] &{}\leq f (x^t) + \langle \nabla f (x^t),
    x^{t + 1} - x^t \rangle + \tfrac{L}{2} \mathbb{E}_t [\| x^{t + 1} - x^t \|^2] \nonumber\\
    &{}= f (x^t) + \langle \nabla f (x^t), - c \, \alpha \nabla f (x^t) +
    \beta_k d^t \rangle + \tfrac{L}{2} \mathbb{E}_t [\| d^{t + 1} \|^2] . 
    \label{eq:pot2-f}
  \end{align}
  By the update rule, $d^{t + 1} = - c \, \alpha g_t + \beta_k d^t$. Take
  conditional expectation and by \eqref{eq:svrg-var},
  \begin{align}
    \mathbb{E}_t [\| d^{t + 1} \|^2] &{}= c^2 \alpha^2 \mathbb{E}_t [\| g_t
    \|^2] + \beta_k^2 \| d^t \|^2 - 2 c \, \alpha \beta_k \langle \nabla f (x^t),
    d^t \rangle \nonumber\\
    &{}\leq 4 c^2 \alpha^2 L (f (x^t) - f^{\star}) + \beta_k^2 \| d^t \|^2 -
    2 c \, \alpha \beta_k \langle \nabla f (x^t), d^t \rangle + 4 c^2 \alpha^2 L
    (f (\tilde{x}^k) - f^{\star}) .  \label{eq:pot2-d}
  \end{align}
  Let $V^t = f (x^t) - f^{\star} + \tfrac{L}{2} \| d^t \|^2$, by
  \eqref{eq:pot2-f} and \eqref{eq:pot2-d},
  \begin{align*}
    \mathbb{E}_t [V^{t + 1}] &{}\leq f (x^t) - f^{\star} + \langle \nabla f
    (x^t), - c \, \alpha \nabla f (x^t) + \beta_k d^t \rangle + L\mathbb{E}_t [\|
    d^{t + 1} \|^2]\\
    &{}\leq f (x^t) - f^{\star} - c \, \alpha \| \nabla f (x^t) \|^2 + (1 - 2 c \,
    \alpha L) \beta_k \langle \nabla f (x^t), d^t \rangle + L \beta_k^2 \| d^t
    \|^2\\
    &{}\quad + 4 c^2 \alpha^2 L^2 (f (x^t) - f^{\star}) + 4 c^2 \alpha^2 L^2 (f
    (\tilde{x}^k) - f^{\star})\\
    &{}= V^t - \left( \tfrac{1}{2} - \beta_k^2 \right) L \| d^t \|^2 - c \,
    \alpha \| \nabla f (x^t) \|^2 + 4 c^2 \alpha^2 L^2 (f (\tilde{x}^k) -
    f^{\star})\\
    &{}\quad + (1 - 2 c \, \alpha L) \beta_k \langle \nabla f (x^t), d^t \rangle + 4
    c^2 \alpha^2 L^2 (f (x^t) - f^{\star})\\
    &{}\leq V^t - \left( \tfrac{1}{2} - \beta_k^2 \right) L \| d^t \|^2 - (c \,
    \alpha - 2 c^2 \alpha^2 L \kappa) \| \nabla f (x^t) \|^2\\
    &{}\quad + (1 - 2 c \, \alpha L) \beta_k \langle \nabla f (x^t), d^t \rangle + 4
    c^2 \alpha^2 L^2 (f (\tilde{x}^k) - f^{\star}).
  \end{align*}
  Thus completes the proof of \eqref{pot-hbm}.
\end{proof}

\subsection{Proof of \cref{prop:svrg-conv}}
\begin{proof}
  \textbf{Part (i) {\svrg}.} Let $c = 1$ and
  $\alpha = \bar{\alpha} = \underline{\alpha}$ in \eqref{pot-svrg} and telescope,
  \begin{align*}
    \mathbb{E} [f (x^m)] &{}\leq f (x^0) + 2 m \alpha^2 L^2 (f (\tilde{x}^k)
    - f^{\star}) - 2 (\alpha \mu - \alpha^2 L^2) \left[ \sum_{t = 0}^{m - 1}
    (\mathbb{E} [f (x^t)] - f (x^{\star})) \right]\\
    &{}= f (\tilde{x}^k) + 2 m \alpha^2 L^2 (f (\tilde{x}^k) - f^{\star}) -
    2 m (\alpha \mu - \alpha^2 L^2) \mathbb{E} [f (\tilde{x}^{k + 1}) -
    f^{\star}] .
  \end{align*}
  Since $\alpha < \tfrac{1}{\kappa L}$, then $\alpha \mu - \alpha^2 L^2 > 0$.
  Rearrange and the contraction ratio is,
  \begin{equation}
    \tfrac{\mathbb{E} [f (\tilde{x}^{k + 1}) - f^{\star}]}{f (\tilde{x}^k) -
    f^{\star}} \leq \tfrac{1 + 2 m \alpha^2 L^2}{2 m (\alpha \mu - 
    \alpha^2 L^2)} . \label{eq:svrg-single}
  \end{equation}
  Telescope \eqref{eq:svrg-single} proves (i). If $\alpha = \tfrac{1}{4 \kappa
  L}$, the contraction ratio is $\tfrac{1}{3} + \tfrac{4 \kappa^2}{3 m}$.
  
  \textbf{Part (ii) {\sh}.} Let $c = 1$ and $\beta_k = {\beta} \leq \bar{\beta}$ in \eqref{pot-hbm}. Suppose $\alpha
  < \tfrac{1}{2 \kappa L}$ and $\bar{\beta} \leq \sqrt{\alpha L - 2 \alpha^2
  L^2 \kappa}$,
  \begin{align*}
    \mathbb{E}_t [V^{t + 1}] &{}\leq V^t - \left( \tfrac{1}{2} -
    \bar{\beta}^2 \right) L \| d^t \|^2 - (\alpha - 2 \alpha^2 L \kappa) \|
    \nabla f (x^t) \|^2\\
    &{}\quad + (1 - 2 \alpha L) \bar{\beta} \langle \nabla f (x^t), d^t \rangle +
    4 \alpha^2 L^2 (f (\tilde{x}^k) - f^{\star})\\
    &{}\leq V^t - \tfrac{L}{4} \| d^t \|^2 - (\alpha - 2 \alpha^2 L \kappa)
    \| \nabla f (x^t) \|^2 + \bar{\beta} | \langle \nabla f (x^t), d^t \rangle
    | + 4 \alpha^2 L^2 (f (\tilde{x}^k) - f^{\star})\\
    &{}\leq V^t - \left( \alpha - 2 \alpha^2 L \kappa -
    \tfrac{\bar{\beta}^2}{L} \right) \| \nabla f (x^t) \|^2 + 4 \alpha^2 L^2
    (f (\tilde{x}^k) - f^{\star})\\
    &{}\leq V^t - \left( 2 \alpha \mu - 4 \alpha^2 L^2 - \tfrac{2
    \bar{\beta}^2}{\kappa} \right) (f (x^t) - f^{\star}) + 4 \alpha^2 L^2 (f
    (\tilde{x}^k) - f^{\star})
  \end{align*}
  \begin{equation}
    \tfrac{\mathbb{E} [f (\tilde{x}^{k + 1}) - f^{\star}]}{f (\tilde{x}^k) -
    f^{\star}} \leq \tfrac{1 + 4 m \alpha^2 L^2}{m\left( 2 \alpha \mu - 4
    \alpha^2 L^2 - \tfrac{2 \bar{\beta}^2}{\kappa} \right)}
    \label{eq:hbm-single}
  \end{equation}
  Telescope \eqref{eq:hbm-single} proves (ii). If $\alpha = \tfrac{1}{8 \kappa
  L}$ and $\bar{\beta} = \tfrac{1}{\sqrt{32 \kappa}}$, the contraction ratio
  is $\tfrac{8 \kappa^2}{m} + \tfrac{1}{2}$.
\end{proof}

\subsection{Proof of \cref{prop:vr-stepsize-bound}}
\begin{proof}
  \textbf{Part (i) Scalar Stepsize.} For hypergradient feedback, the update rule is,
  \[ \alpha_{k + 1} = \alpha_k + \eta \tfrac{\nabla f (x^k - \alpha_k \nabla f
     (x^k))^{\top} \nabla f (x^k)}{\| \nabla f (x^k) \|^2}, \]
  where $x$ could be arbitrary. Let $x^{k+1} = x^k - \alpha_k \nabla f (x^k)$, since
  $\nabla f$ is $\mu$-strongly monotone and $L$-Lipschitz,
  \begin{equation}
    \mu \alpha_k^2 \| \nabla f (x^k) \|^2 \leq \langle x^k - x^{k+1}, \nabla f (x^k) -
    \nabla f (x^{k+1}) \rangle \leq L \alpha_k^2 \| \nabla f (x^k) \|^2
    \label{eq:lip-operator}
  \end{equation}
  Since $\nabla f (x^{k+1})^{\top} \nabla f (x^k) = \| \nabla f (x^k) \|^2 -
  \frac{1}{\alpha_k} \langle x^k - x^{k+1}, \nabla f (x^k) - \nabla f (x^{k+1}) \rangle$,
  we have
  \[ 1 - L \alpha_k \leq \tfrac{\nabla f (x^{k+1})^{\top} \nabla f (x^k)}{\| \nabla
     f (x^k) \|^2} \leq 1 - \mu \alpha_k \]
  Finally, by $\alpha_k \in \left[ \tfrac{1}{L}, \tfrac{1}{\mu} \right]$, we
  derived the bound on $\alpha_{k + 1}$,
  \[ \tfrac{1}{L} \leq (1 - \eta L) \alpha_k + \eta \leq \alpha_k + \eta
     \tfrac{\nabla f (x^{k+1})^{\top} \nabla f (x^k)}{\| \nabla f (x^k) \|^2} \leq (1
     - \eta \mu) \alpha_k + \eta \leq \tfrac{1}{\mu} . \]
  For ratio feedback, the update rule for scalar stepsize is,
  \begin{align*}
    \alpha_{k + 1} &{}= \alpha_k + \eta \tfrac{\nabla f (x^k - \alpha_k \nabla
    f (x^k))^{\top} \nabla f (x^k)}{f (x^k) - f^{\star}}\\
    &{}= \alpha_k + \eta \tfrac{\| \nabla f (x^k) \|^2 - \frac{1}{\alpha_k}
    \langle x^k - x^{k+1}, \nabla f (x^k) - \nabla f (x^{k+1}) \rangle}{f (x^k) - f^{\star}}
  \end{align*}
  By inequality \eqref{eq:lip-operator},
  \[ \alpha_k + \eta \tfrac{(1 - L \alpha_k) \| \nabla f (x^k) \|^2}{f (x^k) -
     f^{\star}} \leq \alpha_{k + 1} \leq \alpha_k + \eta \tfrac{(1 - \mu
     \alpha_k) \| \nabla f (x^k) \|^2}{f (x^k) - f^{\star}} \]
  Since $\alpha_k \in \left[ \frac{1}{L}, \frac{1}{\mu} \right]$, $(1 - L
  \alpha_k) \leq 0$ and $(1 - \mu \alpha_k) \geq 0$. Also, by $\tfrac{\|
  \nabla f (x^k) \|^2}{f (x^k) - f^{\star}} \leq 2 L$,
  \[ (1 - 2 \eta L^2) \alpha_k + 2 \eta L \leq \alpha_{k + 1} \leq (1 - 2 \eta
     L \mu) \alpha_k + 2 \eta L. \]
  When $\eta \leq \tfrac{1}{2 L^2}$, we have $\alpha_{k + 1} \in \left[
  \frac{1}{L}, \frac{1}{\mu} \right]$,
  \[ \tfrac{1}{L} = (1 - 2 \eta L^2)  \tfrac{1}{L} + 2 \eta L \leq \alpha_{k +
     1} \leq (1 - 2 \eta L \mu)  \tfrac{1}{\mu} + 2 \eta L = \tfrac{1}{\mu} .
  \]
  \textbf{Part (ii) Matrix Stepsize.} For hypergradient feedback, the update rule of matrix
  stepsize is,
  \begin{align*}
    \| P_{k + 1} \|_F &{}= \left\| (1 - \eta \rho) P_k + \eta \tfrac{\nabla f
    (x^k - P_k \nabla f (x^k)) \nabla f (x^k)^{\top}}{\| \nabla f (x^k) \|^2}
    \right\|_F\\
    &{}\leq (1 - \eta \rho) \| P_k \|_F + \eta \tfrac{\| \nabla f (x^k - P_k
    \nabla f (x^k)) \|}{\| \nabla f (x^k) \|}\\
    &{}\leq (1 - \eta \rho) \| P_k \|_F + \eta (1 + L \| P_k \|_F)\\
    &{}= (1 - (\rho - L) \eta) \| P_k \|_F + \eta\\
    &{}\leq (1 - (\rho - L) \eta)^{k + 1} \| P_0 \|_F + \tfrac{1}{\rho - L}
  \end{align*}
  where the first inequality is by triangular inequality and the property of
  Fronenius norm, the second inequality is by the $L$-smoothness of $f$, and
  the last inequality is by geometric sum. Given the parameter choice $\rho =
  2 L$, $P_0 = 0$, and $\eta \leq \tfrac{1}{L}$, we have part (ii) hold by $\|
  P_k \|_2 \leq \| P_k \|_F \leq \tfrac{1}{L}$.
  
  For ratio feedback, the update rule of matrix stepsize is,
  \begin{align*}
    \| P_{k + 1} \|_F &{}= \left\| (1 - \eta \rho) P_k + \eta \tfrac{\nabla f
    (x^k - P_k \nabla f (x^k)) \nabla f (x^k)^{\top}}{f (x^k) - f^{\star}}
    \right\|_F\\
    &{}\leq (1 - \eta \rho) \| P_k \|_F + \eta \tfrac{\| \nabla f (x^k - P_k
    \nabla f (x^k)) \| \| \nabla f (x^k) \|}{f (x^k) - f^{\star}}\\
    &{}\leq (1 - \eta \rho) \| P_k \|_F + 2 \eta L \tfrac{\| \nabla f (x^k -
    P_k \nabla f (x^k)) \|}{\| \nabla f (x^k) \|}\\
    &{}\leq (1 - \eta \rho) \| P_k \|_F + 2 \eta L (1 + L \| P_k \|_F)\\
    &{}= (1 - (\rho - 2 L^2) \eta) \| P_k \|_F + 2 \eta L\\
    &{}\leq (1 - (\rho - 2 L^2) \eta)^{k + 1} \| P_0 \|_F + \tfrac{2 L}{\rho
    - 2 L^2}
  \end{align*}
  Let $\rho = 4 L^2$, $P_0 = 0$, and $\eta \leq \tfrac{1}{2 L^2}$, we have
  part (ii) hold by $\| P_k \|_2 \leq \| P_k \|_F \leq \tfrac{1}{L}$.
\end{proof}

\subsection{Proof of \cref{thm:osgm-svrg}}
\begin{proof}
  \textbf{Part (i) No momentum, scalar stepsize.} The update rule is $x^{t +
  1} = x^t - c \alpha_k g_t$. Let $\bar{\alpha} = \underline{\alpha} =
  \alpha_k$ in \eqref{pot-svrg} from \cref{lem:vr-pot},
  \[ \mathbb{E}_t [V^{t + 1}] \leq V^t - (2 c \alpha_k \mu - 2 c^2 \alpha_k^2
     L^2) (f (x^t) - f^{\star}) + 2 c^2 \alpha_k^2 L^2 (f (\tilde{x}^k) -
     f^{\star}) . \]
  Suppose $c \leq \tfrac{1}{\alpha_k L \kappa}$, such that $2 c \alpha_k \mu -
  2 c^2 \alpha_k^2 L^2 \geq 0$. Telescope and rearrange, then the contraction
  factor is,
  \begin{equation}\label{eq:rho1}
    \tfrac{\mathbb{E} [f (\tilde{x}^{k + 1}) - f^{\star}]}{f (\tilde{x}^k) -
    f^{\star}} \leq \tfrac{1 + 2 m c^2 \alpha_k^2 L^2}{2 m (c \alpha_k \mu -
    c^2 \alpha_k^2 L^2)} = \tfrac{1}{2 m (c \alpha_k \mu - c^2 \alpha_k^2
    L^2)} + \tfrac{c^2 \alpha_k^2 L^2}{(c \alpha_k \mu - c^2 \alpha_k^2 L^2)}
    . 
  \end{equation}
  From \cref{prop:vr-stepsize-bound} (i), the stepsize is bounded $\alpha_k \in \left[ \tfrac{1}{L},
  \tfrac{1}{\mu} \right]$, then $c \leq \tfrac{1}{\alpha_k L \kappa}$ implies
  that $c \leq \tfrac{1}{\kappa^2}$. Suppose $c \leq \tfrac{1}{\kappa (\kappa
  + 1)}$, we have the first term of \eqref{eq:rho1} is maximized at $\alpha_k
  = \tfrac{1}{L}$, and the second term of \eqref{eq:rho1} is maximized at
  $\alpha_k = \tfrac{1}{\mu}$. Substitute into \eqref{eq:rho1} gets the upper
  bound of contraction factor,
  \[ \tfrac{\mathbb{E} [f (\tilde{x}^{k + 1}) - f^{\star}]}{f (\tilde{x}^k) -
     f^{\star}} \leq \tfrac{\kappa}{2 m (c - c^2 \kappa)} + \tfrac{c
     \kappa^2}{1 - c \kappa^2} . \]
  If $c = \tfrac{1}{3 \kappa^2}$, we can choose $m = 9 \kappa^3$ to have
  contraction factor $\tfrac{3}{4}$.
  
  \textbf{Part (ii) No momentum, matrix stepsize.} The update rule is $x^{t +
  1} = x^t - c P_k g_t$. From \cref{prop:vr-stepsize-bound} (ii) and the definition of $\mathcal{P}$, we have
  $\underline{\alpha} I \preceq P_k \preceq \tfrac{1}{L} I$. Let
  $\underline{\alpha}, \bar{\alpha} = \tfrac{1}{L}$ in \eqref{pot-svrg} from \cref{lem:vr-pot} (i),
  \[ \mathbb{E}_t [V^{t + 1}] \leq V^t - \left( 2 c \underline{\alpha}
     \mu - 2 c^2 \right) (f (x^t) - f^{\star}) + 2 c^2 (f (\tilde{x}^k) -
     f^{\star}) . \]
  Suppose $c \leq \underline{\alpha} \mu$ such that $2 c \underline{\alpha}
  \mu - 2 c^2 \geq 0$, then the contraction factor is,
  \[ \tfrac{\mathbb{E} [f (\tilde{x}^{k + 1}) - f^{\star}]}{f (\tilde{x}^k) -
     f^{\star}} \leq \tfrac{1 + 2 m c^2}{2 m \left( c \underline{\alpha} \mu -
     c^2 \right)} = \tfrac{1}{2 m \left( c \underline{\alpha} \mu - c^2
     \right)} + \tfrac{c}{\left( \underline{\alpha} \mu - c \right)} . \]
  If $c = \tfrac{\underline{\alpha} \mu}{3}$, we can choose $m =
  \tfrac{9}{\underline{\alpha}^2 \mu^2}$ to have contraction factor
  $\tfrac{3}{4}$.
  
  \textbf{Part (iii) Bounded momentum, scalar stepsize.} The update rule is
  $x^{t + 1} = x^t - c \alpha_k g_t + \beta_k (x^t - x^{t - 1})$. From the
  proof of \cref{lem:vr-pot}, suppose $c \alpha_k < \tfrac{1}{2 \kappa L}$ and $0
  \leq \beta_k \leq \bar{\beta}$, we have
  \begin{equation*}
    \mathbb{E}_t [V^{t + 1}]  \leq  V^t - \left( 2 c \alpha_k \mu - 4 c^2
    \alpha_k^2 L^2 - \tfrac{2 \bar{\beta}^2}{\kappa} \right) (f (x^t) -
    f^{\star}) + 4 c^2 \alpha_k^2 L^2 (f (\tilde{x}^k) - f^{\star}) .
  \end{equation*}
  Suppose $2 \bar{\beta}^2 \leq c \alpha_k \kappa \mu$ and $c \alpha_k \mu - 4
  c^2 \alpha_k^2 L^2 \geq 0$, then the contraction factor is,
  \begin{equation}
    \tfrac{\mathbb{E} [f (\tilde{x}^{k + 1}) - f^{\star}]}{f (\tilde{x}^k) -
    f^{\star}} \leq \tfrac{1 + 4 m c^2 \alpha_k^2 L^2}{m (c \alpha_k \mu - 4
    c^2 \alpha_k^2 L^2)} = \tfrac{1}{m (c \alpha_k \mu - 4 c^2 \alpha_k^2
    L^2)} + \tfrac{4 c \alpha_k L^2}{\mu - 4 c \alpha_k L^2} . \label{eq:rho3}
  \end{equation}
  From \cref{prop:vr-stepsize-bound} (i), the stepsize is bounded $\alpha_k \in \left[ \tfrac{1}{L},
  \tfrac{1}{\mu} \right]$. Suppose $c \leq \tfrac{1}{4 \kappa (\kappa + 1)}$,
  we have the first term of \eqref{eq:rho3} is maximized at $\alpha_k =
  \tfrac{1}{L}$, and the second term of \eqref{eq:rho3} is maximized at
  $\alpha_k = \tfrac{1}{\mu}$. Substitute into \eqref{eq:rho3} gets the upper
  bound of contraction factor,
  \[ \tfrac{\mathbb{E} [f (\tilde{x}^{k + 1}) - f^{\star}]}{f (\tilde{x}^k) -
     f^{\star}} \leq \tfrac{\kappa}{m (c - c^2 \kappa)} + \tfrac{4 c
     \kappa^2}{1 - 4 c \kappa^2} . \]
  If $c = \tfrac{1}{12 \kappa^2}$, which implies that $\bar{\beta} \leq
  \tfrac{1}{2 \sqrt{6} \kappa}$, we choose $m = 72 \kappa^3$ to have
  contraction factor $\tfrac{3}{4}$.
\end{proof}

\section{Experimental details}

\subsection{Dataset details} \label{app:exp-dataset}

\cref{tab:dataset} lists the details of 47 datasets used in \cref{sec:exp}, where $n$ is the number of samples, and $d$ is the number of features. 

\begin{table}[htbp]
  \centering
  \begin{tabularx}{\textwidth}{p{0.2\textwidth}YY|p{0.2\textwidth}YY}
    \toprule
    Dataset & $n$ & $d$ & Dataset & $n$ & $d$ \\
    \midrule
    a1a & 1,605 & 123 & a2a & 2,265 & 123 \\
    a3a & 3,185 & 123 & a4a & 4,781 & 123 \\
    a5a & 6,414 & 123 & a6a & 11,220 & 123 \\
    a7a & 16,100 & 123 & a8a & 22,696 & 123 \\
    a9a & 32,561 & 123 & covtype & 464,809 & 54 \\
    german.numer & 800 & 24 & gisette & 6,000 & 5,000 \\
    ijcnn1 & 35,000 & 22 & madelon & 2,000 & 500 \\
    mushrooms & 6,499 & 112 & news20 & 15,996 & 1,355,191 \\
    phishing & 8,844 & 68 & rcv1 & 20,242 & 47,236 \\
    real-sim & 57,847 & 20,958 & splice & 1,000 & 60 \\
    sonar & 166 & 60 & svmguide3 & 1,243 & 22 \\
    w1a & 2,477 & 300 & w2a & 3,470 & 300 \\
    w3a & 4,912 & 300 & w4a & 7,366 & 300 \\
    w5a & 9,888 & 300 & w6a & 17,188 & 300 \\
    w7a & 24,692 & 300 & w8a & 49,749 & 300 \\
    webspam & 280,000 & 254 & e2006 & 16,087 & 150,360 \\
    yearpredictionmsd & 463,715 & 90 & santander & 160,000 & 200 \\
    miniboone & 104,051 & 50 & guillermo & 16,000 & 4,296 \\
    creditcard & 227,845 & 29 & acsincome & 1,331,600 & 11 \\
    medical & 48,971 & 18 & airlines & 800,000 & 6 \\
    click-prediction & 1,597,928 & 11 & mtp & 3,560 & 202 \\
    elevators & 13,279 & 18 & ailerons & 11,000 & 40 \\
    superconduct & 17,010 & 79 & sarcos & 39,146 & 21 \\
    jannis & 46,064 & 54 &  &  &  \\
    \bottomrule
  \end{tabularx}
  \caption{Details of 47 datasets used in \cref{sec:exp}.}
  \label{tab:dataset}
\end{table}

\subsection{Practical variants} \label{app:exp-algs}

The difference between Practical {\osgdh} (\cref{alg:prac-osgd}) and {\osgd} (\cref{alg:sosgm}) is in the definition of feedback functions. Practical {\osgdh} uses in-sample feedback \eqref{eq:hoverfit} while {\osgd} uses out-of-sample feedback \eqref{eq:def-h}. Also, practical {\osgdh} drops the null step for efficiency.

\begin{algorithm}[ht!]
\caption{{Practical \texttt{OSGM-SGD}}}
\label{alg:prac-osgd}
\begin{algorithmic}[1]
    \STATE \textbf{Input:} Initial iterate $x^0$, initial stepsize $P_0$, candidate stepsize set $\mathcal{P}$, learning rate $\eta$, \\
    feedback $\ell_{x,\xi,\zeta} \in \{ r_{x,\xi,\zeta},\, h_{x,\xi,\zeta} \}$ defined by \eqref{eq:hoverfit}
    \FOR{$k = 0,1,2,\dots$}
        \STATE Sample $\xi^k$ uniformly

        \STATE 
        $x^{k+1}=x^k - P_k \nabla f_{\xi^k}(x^k)$

        \STATE 
        $P_{k+1}=\Pi_{\mathcal{P}} [P_k - \eta \nabla \ell_{x^k,\xi^k,\zeta^k}^k(P_k)]$
    \ENDFOR
\end{algorithmic}
\end{algorithm}

The implementation of practical variants {\ovh} (\cref{alg:prac-osgm-vr}) and {\osk} (\cref{alg:osgm-sketchy}) are inspired by \texttt{OSGM-Best} introduced by \cite{chu2025gradient}, which tune the stepsize and momentum simultaneously,
\[
x^+ = x - P \nabla f(x) +\beta (x - x^-).
\]
Specifically, \cite{chu2025gradient} introduces a joint hypergradient feedback,
\[
h_{x,x^-}(P,\beta) = \tfrac{\phi_{\omega}(x^+(P,\beta), x) - \phi_{\omega}(x, x^-)}{\|\nabla f(x)\|^2 + \tfrac{\tau}{2}\|x - x^-\|^2},
\]
where $\phi_{\omega}(x, x^-)\assign f(x) - f^\star + \tfrac{\omega}{2} \|x - x^-\|^2$ is the potential function for heavy-ball momentum. Then the stepsize and momentum are updated by online gradient descent. In {\ovh} and {\osk}, we substitute the deterministic gradients by their stochastic counterparts.

\begin{algorithm}[ht!]
\caption{Practical {\ovh}}
\label{alg:prac-osgm-vr}
\begin{algorithmic}[1]
\STATE \textbf{Input:} Initial $\tilde{x}^0$, $P_0=0$, $\beta_0=0$, epoch length $m$, {\os} learning rates $\eta_P$ and $\eta_{\beta}$, candidate set of diagonal stepsize and momentum $\mathcal{P}$, $\mathcal{B}$
\FOR{$k=0,1,2,\dots$}
    \STATE Compute snapshot gradient $\nabla f(\tilde{x}^k) = \tfrac1n\sum_{i=1}^n \nabla f_i(\tilde{x}^k)$
    
    \STATE Set $x^0 = \tilde{x}^k$
    \FOR{$t=0,\dots,m-1$}
        \STATE Sample $\xi^t$ uniformly
        \STATE Compute variance reduced gradient
        $
            g_t = \nabla f_{\xi^t}(x^t) - \nabla f_{\xi^t}(\tilde{x}^k) + \nabla f(\tilde{x}^k)
        $
        \vspace{0.1cm}
        
        \STATE 
        $
            x^{t+1}
            =
            x^t
            - P_t g_t
            + \beta_t (x^t - x^{t-1})
        $
        \STATE Define feedback $\ell_t (P,\beta) = \tfrac{f(x^t - P g_t + \beta (x^t - x^{t-1})) - f(x^t)}{\|g_t\|^2 + \|x^t - x^{t-1}\|^2}$

        \STATE Update stepsize:\qquad $P_{t+1} = \Pi_{\mathcal{P}}\big[P_{t} - \eta_P \nabla_P \ell_t (P_t,\beta_t) \big]$
        \vspace{0.1cm}
        \STATE Update momentum: $\beta_{t+1} = \Pi_{\mathcal{B}} \big[ \beta_{t} - \eta_{\beta} \nabla_{\beta} \ell_t (P_t,\beta_t) \big]$
        \vspace{0.1cm}
    \ENDFOR
    \STATE Choose $\tilde{x}^{k+1}$ uniformly from $\{x^0,\dots,x^m\}$, set $P_0=P_m, \beta_0=\beta_m$
\ENDFOR
\end{algorithmic}
\end{algorithm}

\subsection{Additional experiments}

\subsubsection{Evaluations of \cref{alg:osgm-svrg}}\label{app:exp-fig}

We evaluate the performance of {\ov} (\cref{alg:osgm-svrg}) on logistic regression problems. The batchsize is 256.

\paragraph{Benchmark algorithms.} We benchmark the following variance reduction algorithms.
\begin{itemize}
  \item {\svrg} \cite{johnson2013accelerating} with updated frequency $m=\lceil n/256 \rceil$ and tuned stepsize.
  \item {\saga} \cite{defazio2014saga} with tuned stepsize.
  \item {\kat} \cite{kovalev2020don}. \texttt{Loopless Katyusha} \cite{allen2018katyusha} with tuned stepsize.
  \item {\ov} with scalar stepsize, default {\os} learning rate $\eta=1/L$ and tuned constant momentum $\beta \in \{0.0, 0.9, 0.98\}$
\end{itemize}

\paragraph{Performance plots.} \cref{fig:epoch-logistic} shows the suboptimality plots on 8 medium-sized logistic regression problems. Notice that {\ov} with default learning rate outperforms baseline algorithms with tuned stepsize, especially in the later iterations. This experiment suggests the effectiveness of {\ov} as an outer-loop stepsize scheduler.

\begin{figure}[htbp]
  \centering
  \includegraphics[width=\linewidth]{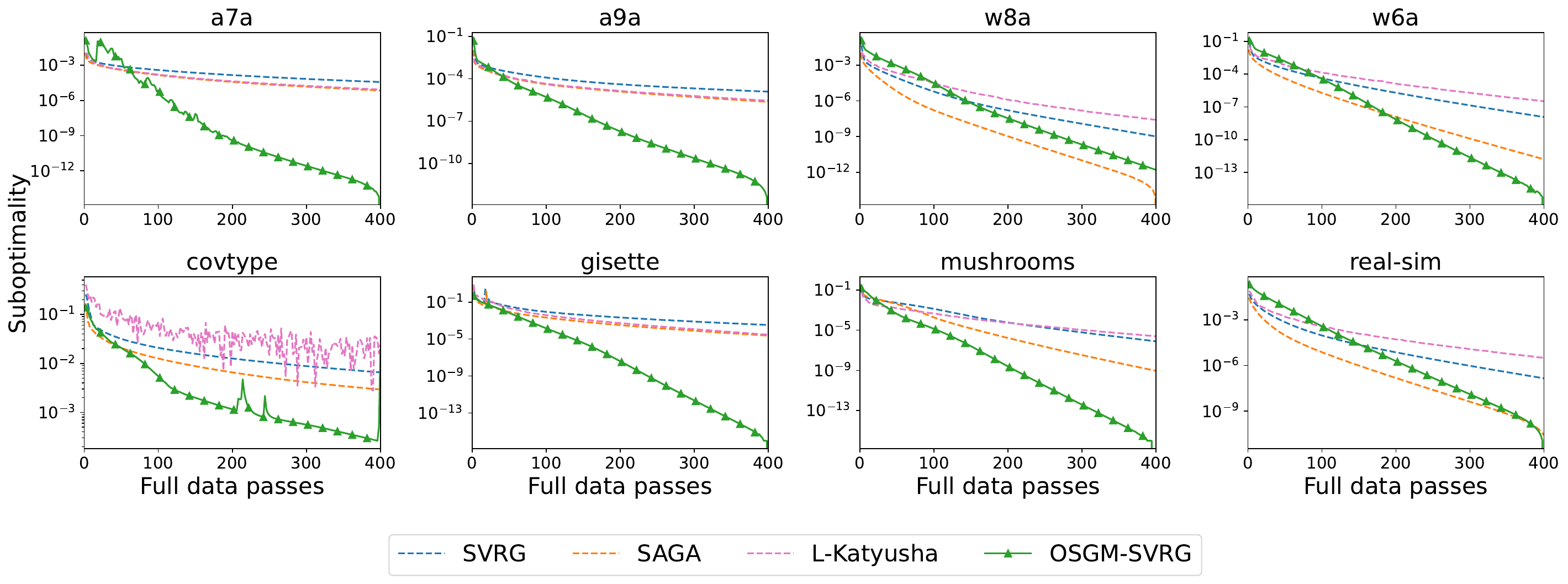}
  \caption{Performance of \cref{alg:osgm-svrg} on logistic regression problems.}
  \label{fig:epoch-logistic}
\end{figure}

\subsubsection{Evaluations of \cref{alg:sosgm}}\label{app:exp-fig2}

We evaluate the performance of \cref{alg:sosgm} with feedback \eqref{eq:hoverfit} (labeled as Naive) and feedback \eqref{eq:def-h} (labeled as Theory). 
The experiment setting is the same as \cref{sec:exp}. For benchmark algorithms, we use {\sgd} with stepsize $10^{-3}$, and {\osgd} with initial stepsize $10^{-3}$, {\os} learning rate $10^{-3}$.

From the performance plot \cref{fig:osgd-feedbacks}, {\osgd} outperforms {\sgd}, and {\osgdh} outperforms {\osgd}. As suggested in \cref{sec:large-batchsize}, {\osgdh} \eqref{eq:hoverfit} may not converge in some cases. In practice, however, using feedback \eqref{eq:hoverfit} may result in more aggressive stepsize tuning and faster convergence.

\begin{figure}[htbp]
  \centering
  \includegraphics[width=\linewidth]{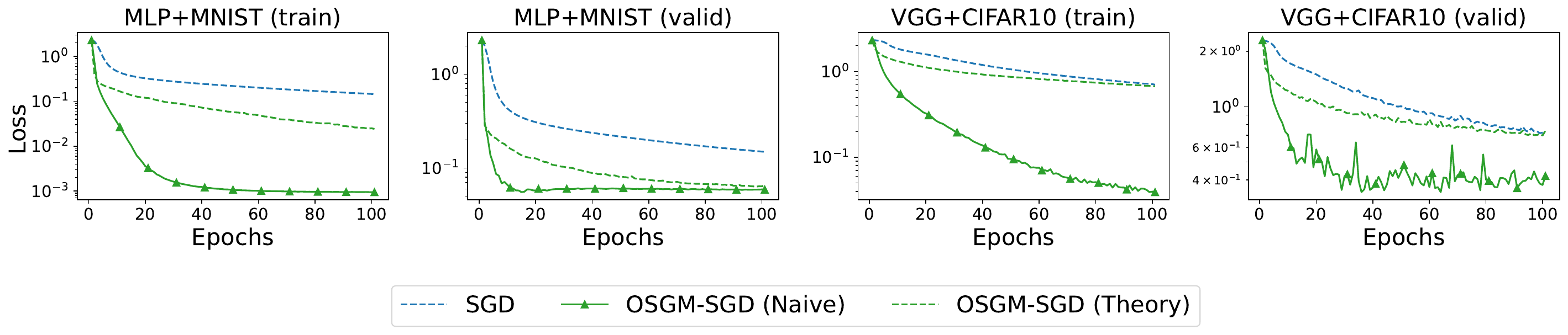}
  \caption{Performance of \cref{alg:sosgm} with feedback \eqref{eq:hoverfit} and feedback \eqref{eq:def-h}.}
  \label{fig:osgd-feedbacks}
\end{figure}

\end{document}